\documentclass[12pt,dvipsnames]{amsart}

\usepackage{amssymb,mathdots,mathtools,nccmath}

\usepackage[T1]{fontenc}
\usepackage[utf8]{inputenc}

\usepackage[aligntableaux = center]{ytableau}
\usepackage{graphicx}
\usepackage[margin=2.5cm]{geometry}
\usepackage{enumitem}
\usepackage{xcolor}
\usepackage{booktabs}
\usepackage{makecell}
\usepackage{nicematrix}
\usepackage[normalem]{ulem}

\usepackage{tikz}
\usetikzlibrary{arrows,decorations.pathreplacing}
\newcommand{\midarrow}{\tikz \draw[-triangle 90] (0,0) -- +(.1,0);}

\newcommand\myatop[2]{\genfrac{}{}{0pt}{}{#1}{#2}}

\newcommand{\myscale}{0.75}

\newcommand{\half}{1\mkern-2mu/\mkern-2mu2}

\usepackage[sc]{mathpazo}

\usepackage{subcaption}
\newtheorem{thm}{Theorem}[section]
\newtheorem{prop}[thm]{Proposition}
\newtheorem{cor}[thm]{Corollary}
\newtheorem{lem}[thm]{Lemma}
\newtheorem{conj}[thm]{Conjecture}
\newtheorem{prob}[thm]{Problem}
\newtheorem{dfn}[thm]{Definition}
\newtheorem{exm}[thm]{Example}

\newcommand{\nenwarrow}{\nwarrow \!\!\!\!\!\;\!\! \nearrow}

\newcommand{\deff}[1]{\emph{#1}}

\newcommand{\DSASMSym}{%
	\tikz[scale=0.12]{%
		\draw (0,0) -- (2,0) -- (2,-2) -- cycle;
	}%
}

\DeclareMathOperator{\sgn}{sgn}
\DeclareMathOperator{\asym}{\mathbf{ASym}}
\DeclareMathOperator{\DSASM}{DSASM}
\DeclareMathOperator{\SV}{6V_{\DSASMSym}}
\DeclareMathOperator*{\pf}{Pf}

\renewcommand{\le}{\leqslant}
\renewcommand{\ge}{\geqslant}

\newcommand{\imag}{\mathfrak{i}}
\newcommand{\F}{\mathsf{F}}

\newcommand{\boldk}{\boldsymbol{k}}

\newcommand{\wone}{\scalebox{.6}{\begin{tikzpicture}[scale=.65,every node/.style={sloped,allow upside down}]
			\node at (0,0) {$\bullet$};
			\draw[very thick] (0,0) --node{\midarrow} (0,1);
			\draw[very thick] (0,0) --node{\midarrow} (0,-1);
			\draw[very thick] (-1,0) --node{\midarrow} (0,0);
			\draw[very thick] (1,0) --node{\midarrow} (0,0);
\end{tikzpicture}}}

\newcommand{\wmone}{\scalebox{.6}{\begin{tikzpicture}[scale=.65,every node/.style={sloped,allow upside down}]
			\node at (0,0) {$\bullet$};
			\draw[very thick] (0,1) --node{\midarrow} (0,0);
			\draw[very thick] (0,-1) --node{\midarrow} (0,0);
			\draw[very thick] (0,0) --node{\midarrow} (-1,0);
			\draw[very thick] (0,0) --node{\midarrow} (1,0);
\end{tikzpicture}}}

\newcommand{\wne}{\scalebox{.6}{\begin{tikzpicture}[scale=.65,every node/.style={sloped,allow upside down}]
			\node at (0,0) {$\bullet$};
			\draw[very thick] (0,0) --node{\midarrow} (0,1);
			\draw[very thick] (0,-1) --node{\midarrow} (0,0);
			\draw[very thick] (-1,0) --node{\midarrow} (0,0);
			\draw[very thick] (0,0) --node{\midarrow} (1,0);
\end{tikzpicture}}}

\newcommand{\wneC}{\scalebox{.6}{\begin{tikzpicture}[scale=.65,every node/.style={sloped,allow upside down}]
			\draw[very thick] (0,0) --node{\midarrow} (0,1);
			\draw[very thick] (0,-1) --node{\midarrow} (0,0);
			\draw[very thick] (-1,0) --node{\midarrow} (0,0);
			\draw[very thick] (0,0) --node{\midarrow} (1,0);
			\node at (0,0) {\color{cyan}{$\bullet$}};
\end{tikzpicture}}}

\newcommand{\wse}{\scalebox{.6}{\begin{tikzpicture}[scale=.65,every node/.style={sloped,allow upside down}]
			\node at (0,0) {$\bullet$};
			\draw[very thick] (0,1) --node{\midarrow} (0,0);
			\draw[very thick] (0,0) --node{\midarrow} (0,-1);
			\draw[very thick] (-1,0) --node{\midarrow} (0,0);
			\draw[very thick] (0,0) --node{\midarrow} (1,0);
\end{tikzpicture}}}

\newcommand{\wseC}{\scalebox{.6}{\begin{tikzpicture}[scale=.65,every node/.style={sloped,allow upside down}]
			\draw[very thick] (0,1) --node{\midarrow} (0,0);
			\draw[very thick] (0,0) --node{\midarrow} (0,-1);
			\draw[very thick] (-1,0) --node{\midarrow} (0,0);
			\draw[very thick] (0,0) --node{\midarrow} (1,0);
			\node at (0,0) {\color{teal}{$\bullet$}};
\end{tikzpicture}}}

\newcommand{\wsw}{\scalebox{.6}{\begin{tikzpicture}[scale=.65,every node/.style={sloped,allow upside down}]
			\node at (0,0) {$\bullet$};
			\draw[very thick] (0,1) --node{\midarrow} (0,0);
			\draw[very thick] (0,0) --node{\midarrow} (0,-1);
			\draw[very thick] (0,0) --node{\midarrow} (-1,0);
			\draw[very thick] (1,0) --node{\midarrow} (0,0);
\end{tikzpicture}}}

\newcommand{\wswC}{\scalebox{.6}{\begin{tikzpicture}[scale=.65,every node/.style={sloped,allow upside down}]
			\draw[very thick] (0,1) --node{\midarrow} (0,0);
			\draw[very thick] (0,0) --node{\midarrow} (0,-1);
			\draw[very thick] (0,0) --node{\midarrow} (-1,0);
			\draw[very thick] (1,0) --node{\midarrow} (0,0);
			\node at (0,0) {\color{cyan}{$\bullet$}};
\end{tikzpicture}}}

\newcommand{\wnw}{\scalebox{.6}{\begin{tikzpicture}[scale=.65,every node/.style={sloped,allow upside down}]
			\node at (0,0) {$\bullet$};
			\draw[very thick] (0,0) --node{\midarrow} (0,1);
			\draw[very thick] (0,-1) --node{\midarrow} (0,0);
			\draw[very thick] (0,0) --node{\midarrow} (-1,0);
			\draw[very thick] (1,0) --node{\midarrow} (0,0);
\end{tikzpicture}}}

\newcommand{\wnwC}{\scalebox{.6}{\begin{tikzpicture}[scale=.65,every node/.style={sloped,allow upside down}]
			\draw[very thick] (0,0) --node{\midarrow} (0,1);
			\draw[very thick] (0,-1) --node{\midarrow} (0,0);
			\draw[very thick] (0,0) --node{\midarrow} (-1,0);
			\draw[very thick] (1,0) --node{\midarrow} (0,0);
			\node at (0,0) {\color{teal}{$\bullet$}};
\end{tikzpicture}}}

\newcommand{\wlone}{\scalebox{.6}{\begin{tikzpicture}[scale=.65,every node/.style={sloped,allow upside down}]
			\node at (0,0) {$\bullet$};
			\draw[very thick] (0,0) --node{\midarrow} (0,1);
			\draw[very thick] (1,0) --node{\midarrow} (0,0);
\end{tikzpicture}}}

\newcommand{\wloneC}{\scalebox{.6}{\begin{tikzpicture}[scale=.65,every node/.style={sloped,allow upside down}]
			\draw[very thick] (0,0) --node{\midarrow} (0,1);
			\draw[very thick] (1,0) --node{\midarrow} (0,0);
			\node at (0,0) {\color{orange}{$\bullet$}};
\end{tikzpicture}}}

\newcommand{\wlmone}{\scalebox{.6}{\begin{tikzpicture}[scale=.65,every node/.style={sloped,allow upside down}]
			\node at (0,0) {$\bullet$};
			\draw[very thick] (0,1) --node{\midarrow} (0,0);
			\draw[very thick] (0,0) --node{\midarrow} (1,0);
\end{tikzpicture}}}

\newcommand{\wlout}{\scalebox{.6}{\begin{tikzpicture}[scale=.65,every node/.style={sloped,allow upside down}]
			\node at (0,0) {$\bullet$};
			\draw[very thick] (0,0) --node{\midarrow} (0,1);
			\draw[very thick] (0,0) --node{\midarrow} (1,0);
\end{tikzpicture}}}

\newcommand{\wloutC}{\scalebox{.6}{\begin{tikzpicture}[scale=.65,every node/.style={sloped,allow upside down}]
			\draw[very thick] (0,0) --node{\midarrow} (0,1);
			\draw[very thick] (0,0) --node{\midarrow} (1,0);
			\node at (0,0) {\color{violet}{$\bullet$}};
\end{tikzpicture}}}

\newcommand{\wlin}{\scalebox{.6}{\begin{tikzpicture}[scale=.65,every node/.style={sloped,allow upside down}]
			\node at (0,0) {$\bullet$};
			\draw[very thick] (0,1) --node{\midarrow} (0,0);
			\draw[very thick] (1,0) --node{\midarrow} (0,0);
\end{tikzpicture}}}

\newcommand{\wlinC}{\scalebox{.6}{\begin{tikzpicture}[scale=.65,every node/.style={sloped,allow upside down}]
			\draw[very thick] (0,1) --node{\midarrow} (0,0);
			\draw[very thick] (1,0) --node{\midarrow} (0,0);
			\node at (0,0) {\color{orange}{$\bullet$}};
\end{tikzpicture}}}

\newcommand{\wL}{\scalebox{.6}{\begin{tikzpicture}[scale=.65]
			\draw[very thick] (0,1) -- (0,0);
			\draw[very thick] (1,0) -- (0,0);
\end{tikzpicture}}}

\usepackage[colorlinks=true,hyperfootnotes=true]{hyperref}

\title{A Littlewood-type identity for Robbins polynomials}

\author[I. Fischer]{Ilse Fischer}
\address{Faculty of Mathematics, University of Vienna, Austria}
\email{\href{mailto:ilse.fischer@univie.ac.at}{ilse.fischer@univie.ac.at}}
\urladdr{\href{https://www.mat.univie.ac.at/~ifischer}{https://www.mat.univie.ac.at/~ifischer}}

\author[H. Höngesberg]{Hans Höngesberg}
\address{Faculty of Mathematics, University of Vienna, Austria}
\email{\href{mailto:hans.hoengesberg@univie.ac.at}{hans.hoengesberg@univie.ac.at}}
\urladdr{\href{https://hoengesberg.com}{https://hoengesberg.com}}

\thanks{I. Fischer acknowledges support from the Austrian Science Fund (FWF) grant 10.55776/F1002. H. Höngesberg acknowledges support from the FWF grants 10.55776/P34931 and 10.55776/J4810. Part of this research was performed while H. Höngesberg was visiting the Institute for Pure and Applied Mathematics (IPAM), which is supported by the National Science Foundation (Grant No. DMS-1925919).}

\begin{document}
	
\begin{abstract}
	We provide a generalization of the Littlewood  identity, both sides of which are related to alternating sign matrices. The classical Littlewood identity establishes a nice product formula for the sum of all Schur polynomials. Compared to the classical identity, Schur polynomials are replaced by so-called \deff{modified Robbins polynomials}. These polynomials are a generalization of Schur polynomials and enumerate \deff{down-arrowed monotone triangles}, and thus also alternating sign matrices. As an additional factor on the other side of the identity, we have a Pfaffian formula which we interpret in terms of the partition function of six-vertex model configurations corresponding to \deff{diagonally symmetric alternating sign matrices}.
\end{abstract}

\maketitle
	
\section{Introduction}

The classical Littlewood identity\footnote{In fact, there are three different sums of Schur polynomials known as Littlewood identities. The most prominent one is the identity shown here. The other two identities provide factorizations for the sum of Schur polynomials over all partitions with even parts and over all partitions whose conjugates have even parts, respectively.} \cite{Lit50} states that the sum of Schur polynomials $s_{\lambda} (x_1,\allowbreak \dots,\allowbreak x_n)$ over all integer partitions $\lambda$ admits a surprisingly nice factorization:
\begin{equation}
	\label{eq:Littlewood}
	\sum_{\lambda} s_{\lambda} (x_1,\dots,x_n) = \prod_{i=1}^{n} \frac{1}{1-x_i} \prod_{1 \le i < j \le n} \frac{1}{1-x_i x_j}.
\end{equation}
The Littlewood identity is besides the Cauchy identity one of the most influential identities in the theory of symmetric functions. Over the years, several generalizations and variants of the Littlewood identity have been found. It has various applications in representation theory as well as combinatorics, for instance, in the theory of plane partitions and alternating sign matrices (ASMs) (see, among others, \cite{Bre00,BW16,BWZJ15}). The first connection of such identities to ASMs was established by Warnaar in \cite{War08}, where a partition function of six-vertex model configurations corresponding to ASMs appears on the right-hand side of a Cauchy-type identity of Hall--Littlewood polynomials.

In this paper, we present a Littlewood-type identity for modified Robbins polynomials $R^{\ast}_{\boldk}(x_{1},\ldots,x_{n};u,v,w)$, which are three-parameter generalizations of Schur polynomials occurring as the generating function of down-arrowed monotone triangles (DAMTs); see 
Sections~\ref{sec:robbins} and \ref{sec:down} for detailed definitions. These objects have recently been introduced by Schreier-Aigner and the first author as a generalization of ASMs. In particular, we establish the following theorem, which we will prove in Section~\ref{sec:mainproof}.
\begin{thm}
	\label{thm:Robbins_Littlewood}
	Let $n$ be a positive integer. Then 
	\begin{multline}\label{eq:Robbins_Littlewood}
		\sum_{0 \le k_1 < \dots < k_n} R^{\ast}_{(k_{1},\ldots,k_{n})}(x_{1},\ldots,x_{n};1,1,w)\\
		=\prod_{i=1}^n \frac{1}{1-x_i}  \prod_{1 \le i < j \le n} \frac{x_i+x_j+w x_i x_j}{x_j-x_i}
		\pf_{\chi_{\mathrm{even}}(n) \le i < j \le n} \left(
		\begin{cases}
			1, & i=0,\\
			\frac{(x_j-x_i)(1+(1+w) x_i x_j)}{(x_i+x_j+w x_i x_j)(1-x_i x_j)},& i\ge 1,
		\end{cases} \right)
	\end{multline}
	where $\pf$ denotes the Pfaffian of an upper triangular array and $\chi_{\mathrm{even}}(n)$ equals $1$ if $n$ is even and $0$ otherwise.
\end{thm}

By straightforward algebraic manipulation of \eqref{eq:Robbins_Littlewood}, we obtain another Littlewood-type identity for modified Robbins polynomials where we sum over weakly decreasing indices instead of strictly increasing indices.

\begin{cor}\label{cor:Littlewood}
	Let $n$ be a positive integer. Then 
	\begin{multline*}
		\sum_{k_1 \ge \dots \ge k_n \ge 0} R^{\ast}_{(k_{1},\ldots,k_{n})}(x_{1},\dots,x_{n};1,1,w)\\
		=  \prod_{i=1}^n \frac{1}{1-x_i}  \prod_{1 \le i < j \le n} \frac{x_i+x_j+w}{x_i-x_j}
		\pf_{\chi_{\mathrm{even}}(n) \le i < j \le n} \left(
		\begin{cases}
			1, & i=0,\\
			\frac{(x_j-x_i)((1+w) + x_i x_j)}{(x_i+x_j+w)(1-x_i x_j)},& i\ge 1.
		\end{cases} \right)%
	\end{multline*}
\end{cor}

While DAMTs are defined for strictly increasing~$\boldk$, there also exists a signed combinatorial interpretation underlying Corollary~\ref{cor:Littlewood}, which we briefly mention in Section~\ref{sec:down}.

Notably, Theorem~\ref{thm:Robbins_Littlewood} and Corollary~\ref{cor:Littlewood} exhibit combinatorial reciprocity as
\begin{multline*}
	\sum_{k_1 \ge \dots \ge k_n \ge 0} R^{\ast}_{(k_{1},\ldots,k_{n})}(x_{1},\ldots,x_{n};1,1,w)\\
	= (-1)^n \prod_{i=1}^{n} x_i^{-1} \sum_{0 \le k_1 < \dots < k_n} R^{\ast}_{(k_{1},\ldots,k_{n})}(x_{1}^{-1},\ldots,x_{n}^{-1};1,1,w).
\end{multline*}
Combinatorial reciprocity phenomena for monotone triangles have been studied before, see for instance 
\cite{FisRi2013}.

We also provide a combinatorial interpretation of the right-hand side of \eqref{eq:Robbins_Littlewood} in terms of the partition function~$Z_{\DSASM}$ for six-vertex model configurations corresponding to diagonally symmetric alternating sign matrices (DSASMs); see Section~\ref{sec:six-vertex} for further details.

\begin{thm}
	\label{thm:comb_interpretation_RHS}
	Let $n$ be a positive integer. Then 
	\begin{equation*}
		\sum_{0 \le k_1 < \dots < k_n} R^{\ast}_{(k_{1},\ldots,k_{n})}(x_{1},\ldots,x_{n};1,1,w)
		= \prod_{i=1}^{n} \frac{1}{1-x_i} \prod_{1 \le i < j \le n} \frac{1}{1-x_i x_j} Z_{\DSASM} (x_1,\dots,x_n).
	\end{equation*}
\end{thm} 

The main purpose of Theorems~\ref{thm:Robbins_Littlewood} and \ref{thm:comb_interpretation_RHS} is twofold: on the one hand, to establish a Littlewood-type identity for modified Robbins polynomials by presenting a closed expression for $\sum_{0 \le k_1 < \dots < k_n} R^{\ast}_{(k_{1},\ldots,k_{n})}(x_{1},\ldots,x_{n};1,1,w)$, and on the other hand, to provide a combinatorial interpretation of this expression in terms of certain six-vertex model configurations. Although we currently lack such an interpretation for the left-hand side of the generalized Littlewood identity~\eqref{eq:Robbins_Littlewood},
both theorems reveal a deep connection to \emph{alternating sign matrices} (ASMs),  which are $n \times n$ matrices with entries $-1$, $0$ and $1$ such that the nonzero entries along each row and column alternate in sign and sum to one.
Since the introduction of ASMs in the early 1980s by Robbins and Rumsey \cite{RR86}, the problem of enumerating ASMs was challenging combinatorialists for over a decade. It was finally resolved by Zeilberger \cite{Zei96}, and a few other proofs have followed since then \cite{Kup96,Fis07}. The enumeration of symmetry classes of ASMs, initiated by Richard Stanley \cite{Rob91,Rob}, has been an elusive problem ever since. There are eight different symmetry classes of ASMs that are induced by the symmetry group of the square. Among them, the enumeration of \deff{diagonally symmetric alternating sign matrices}, that is, ASMs that stay invariant under matrix transposition, stands out. It is the first symmetry class for which no product formula is known or conjectured but for which an enumerative formula of a different kind has been established. This has recently been achieved in \cite[Corollary~4]{BFK23} by providing the following Pfaffian formula for the number of DSASMs of order~$n$:
\begin{equation}\label{eq:DSASMs}
	\pf_{\chi_{\mathrm{odd}}(n) \le i < j \le n-1} \left(\langle u^i v^j\rangle \frac{(v-u)(2+u v)}{(1-u v)(1- u-v)} \right),
\end{equation}
where, analogous to before, $\chi_{\mathrm{odd}}(n)$ equals $1$ if $n$ is odd and $0$ otherwise; furthermore, $\langle u^i v^j \rangle f(u,v)$ denotes the coefficient of $u^i v^j$ in the expansion of $f(u,v)$.

Our third main theorem provides another connection between the unbounded sum of modified Robbins polynomials and ASMs. We show that the coefficient of the highest term in $Z_{\DSASM} (x_1,\dots,x_n)$ yields a weighted enumeration of DSASMs; see Section~\ref{sec:highest_term} for further details.

\begin{thm}\label{thm:highest_term}
	The coefficient of $x_1^{n-1} \cdots x_n^{n-1}$ in $Z_{\DSASM} (x_1,\dots,x_n)$ is given by the generating function
	\begin{equation*}
		\sum_{\SV(n)} w^{\# \raisebox{-1ex}{\scalebox{.75}{\wnw}}+\# \raisebox{-1ex}{\scalebox{.75}{\wse}}}.
	\end{equation*}
	In particular, the coefficient is
	\begin{equation*}
		w^{\binom{n}{2}} \pf_{\chi_{\mathrm{odd}}(n) \le i < j \le n-1} \left(\langle u^i v^j\rangle \frac{(v-u)(1+u v+w)}{(1-u v)(w- u-v)} \right),
	\end{equation*}
	which coincides with the enumeration formula~\eqref{eq:DSASMs} for DSASMs of order~$n$ if $w=1$.
\end{thm}

We discuss aspects of the relevance of the result. To date, we know three other classes of combinatorial objects that are equinumerous with ASMs: \deff{descending plane partitions} (DPPs), \deff{totally symmetric self-complementary plane partitions} (TSSCPPs) and \deff{alternating sign triangles} (ASTs). 
See \cite{And79,MRR83,MRR86,And94,ABF20} for parts of the relevant literature.
However, no bijection between these classes of objects has been found so far. The lack of bijective proofs has become one of the main driving forces for research in this area. See \cite{BP99,FK20} for some concise overviews over the history of the ASM theorem and the missing bijections.

The classical Littlewood identity admits a bijective proof by exploiting the symmetry of the celebrated \emph{Robinson--Schensted--Knuth correspondence} (RSK). A bijective proof of  Littlewood-type identities related to ASMs might lead to bijections between the four different classes of combinatorial objects mentioned above. This is all the more promising since Littlewood-type identities have played a key role in relating alternating sign matrix objects (ASMs and ASTs) with plane partition objects (DPPs and TSSCPPs) in a computational manner, see for instance \cite{Fis19, Fis19Trap, Hon21, Hon22}. 
A bounded version \cite{Fis22} was used in a similar vein in \cite{gangl}.

However, the significance of our main results goes beyond presenting a concrete idea of how bijections between ASMs and plane partitions could be constructed: a bounded version of Theorem~\ref{thm:Robbins_Littlewood} would be a key ingredient in solving the long-standing Gog-Magog Conjecture \cite{Kra96}, \cite[Conjecture~5]{Kra16}, which has remained unsolved for nearly three decades. On a different note, since the classical Littlewood identity is 
related to the character theory of classical groups, the generalization also raises the question of whether it has a representation-theoretic meaning. 

\subsection*{Outline of the paper}
In Section~\ref{sec:robbins}, we introduce the modified Robbins polynomials $R_{\boldk}^{\ast}(x_{1},\ldots,x_{n};u,v,w)$ and explain the correlation to other symmetric functions. We pre\-sent DAMTs as a combinatorial model for $R_{\boldk}^{\ast}(x_{1},\ldots,x_{n};u,v,w)$ in Section~\ref{sec:down} and provide two different proofs of Theorem~\ref{thm:Robbins_Littlewood} in Section~\ref{sec:mainproof}, where we also derive Corollary~\ref{cor:Littlewood}. In Section~\ref{sec:six-vertex}, we introduce a six-vertex model for DSASMs and establish the Pfaffian formula for the corresponding partition function $Z_{\DSASM} (x_1,\dots,x_n)$, proving Theorem~\ref{thm:comb_interpretation_RHS}. Then, in Section~\ref{sec:highest_term}, we show that the highest term in the polynomial expansion of $Z_{\DSASM} (x_1,\dots,x_n)$ is given by a generating function of DSASMs (Theorem~\ref{thm:highest_term}). The computational proof of Theorem~\ref{thm:highest_term} equates two generating functions of DSASMSs; we pose it as an open problem (Problem~\ref{prob:bijective_proof}) to prove this equality bijectively. In Section~\ref{sec:recovery}, we explain how to recover the classical Littlewood identity from the Littlewood identity for modified Robbins polynomials and present a related open problem. We conclude the paper with Section~\ref{sec:conjecture}, in which we propose a generalization of Theorem~\ref{thm:Robbins_Littlewood} for fully inhomogeneous spin Hall--Littlewood symmetric rational functions and discuss several of its implications.

\section{Modified Robbins polynomials}
\label{sec:robbins}

To begin with, let us fix some notation. For any positive integer~$n$, we denote by $\mathfrak{S}_{n}$ the symmetric group of order~$n$. We define the \deff{antisymmetrizer} $\asym_{x_1,\ldots,x_n}$ to be 
\begin{equation*}
	\asym_{x_1,\ldots,x_n} F(x_{1},\ldots,x_{n}) \coloneq \sum_{\sigma \in \mathfrak{S}_{n}}\sgn(\sigma) F(x_{\sigma(1)},\ldots,x_{\sigma(n)}).
\end{equation*}
The antisymmetrizer is needed to define the modified Robbins polynomials, which are one of the central objects of this paper.

\begin{dfn}
	Let $\boldk=(k_{1},\ldots,k_{n})$ be a sequence of integers. We define the \deff{modified Robbins polynomials} $R_{\boldk}^{\ast}(x_{1},\ldots,x_{n};u,v,w)$ to be
	\begin{equation*}
		\frac{\asym_{x_{1},\ldots,x_{n}}\left[\prod\limits_{1\le i<j\le n}(ux_{i}x_{j}+v+wx_{i})\prod\limits_{i=1}^{n}x_{i}^{k_{i}}\right]}{\prod\limits_{1\le i<j\le n}(x_{j}-x_{i})}.
	\end{equation*}
\end{dfn}

We can recover the Schur polynomials from the modified Robbins polynomials in several ways. We state two of them in the following proposition; in this proposition, we set $s_\lambda(x_1,\ldots,x_n)=0$ if $\lambda$ is not weakly decreasing. Both assertions are a direct consequence of the bialternant formula for Schur polynomials, which states---in an uncommon formulation using the antisymmetrizer---that
\begin{equation*}
	s_{(k_n,\dots,k_1)}(x_1,\dots,x_n)=\frac{\asym_{x_{1},\ldots,x_{n}}\left[\prod\limits_{i=1}^{n}x_{i}^{k_{i}+i-1}\right]}{\prod\limits_{1\le i<j\le n}(x_{j}-x_{i})}.
\end{equation*}

\begin{prop}
	\label{prop:RobbinsToSchur}
	Let $(k_{1},\ldots,k_{n})$ be a strictly increasing sequence of integers. Then it holds that
	\begin{enumerate}[label=(\roman*)]
		\item\label{prop:RobbinsToSchur1}
		$\begin{aligned}[t]
			R_{(k_{1},\ldots,k_{n})}^{\ast}(x_{1},&\ldots,x_{n};0,0,1)\\
			&= s_{(k_n-n+1,k_{n-1}-n+3,\dots,k_2+n-3,k_1+n-1)}(x_{1},\ldots,x_{n})\\
			&=s_{(k_n-2(n-1),k_{n-1}-2(n-2),\dots,k_2-2,k_1)}(x_{1},\ldots,x_{n}) \prod_{i=1}^{n} x_i^{n-1};
		\end{aligned}$		
		\item\label{prop:RobbinsToSchur2}
		$\begin{aligned}[t] R_{(k_{1},\ldots,k_{n})}^{\ast}(x_{1},&\ldots,x_{n};1,1,0)\\
			&= s_{(k_n-n+1,k_{n-1}-n+2,\dots,k_2-1,k_1)}(x_{1},\ldots,x_{n}) \prod\limits_{1 \le i < j \le n} (1+x_i x_j).
		\end{aligned}$
	\end{enumerate}
\end{prop}

The $K$-theoretic analogues of Schur polynomials are known as \deff{Grothendieck polynomials}. Modified Robbins polynomials generalize symmetric Grothendieck polynomials $G_{\lambda} (x_1,\dots,x_n;\beta)$ as can be seen from their bialternant formula \cite{IN13}:
\begin{equation*}
	G_{\lambda} (x_1,\dots,x_n;\beta) = \frac{\det\limits_{1 \le i,j \le n} \left(x_i^{\lambda_j+n-j} (1+ \beta x_i)^{j-1}\right)}{\prod\limits_{1\le i < j \le n} (x_i - x_j)}.
\end{equation*}
Symmetric Grothendieck polynomials $G_{\lambda} (x_1,\dots,x_n;\beta)$ are stable Grothendieck polynomials indexed by Grassmannian permutations.

\begin{prop}
	Let $(k_{1},\ldots,k_{n})$ be a strictly increasing sequence of integers. Then
	\begin{equation*}
		R_{\lambda}^{\ast}(x_{1},\ldots,x_{n};0,1,\beta) = G_{(k_n-n+1,k_{n-1}-n+2,\ldots,k_2-1,k_1)} (x_1,\dots,x_n;\beta).
	\end{equation*}
\end{prop}

Theorem~\ref{thm:Robbins_Littlewood} establishes a closed formula for the sum of all modified Robbins polynomials~$R^{\ast}_{(k_{1},\ldots,k_{n})}(x_{1},\ldots,x_{n};1,1,w)$. There are several other identities in the literature reminiscent of \eqref{eq:Robbins_Littlewood}. In \cite{BWZJ15}, Betea, Wheeler and Zinn-Justin proved a Littlewood-type identity for Hall--Littlewood polynomials which had earlier been conjectured in \cite{BW16}. Later, this identity was generalized by introducing an additional parameter \cite{WZJ16}. Gavrilova \cite{Gav23} established another generalization for \deff{fully inhomogeneous spin Hall--Littlewood symmetric rational functions}:

The fully inhomogeneous spin Hall--Littlewood symmetric rational functions~$ \F_{\lambda}(u_1,\allowbreak \dots,\allowbreak u_n)$ were introduced in \cite{BP18} and admit the following antisymmetrizer formula
\begin{equation*}
	\frac{\asym_{u_{1},\ldots,u_{n}}\left[\prod\limits_{1\le i<j\le n}(u_i-q u_j)\prod\limits_{i=1}^{n} \left(\frac{1-q}{1-s_{\lambda_i} \xi_{\lambda_i} u_i} \prod\limits_{j=0}^{\lambda_i-1} \frac{\xi_j u_i - s_j}{1 - \xi_j s_j u_i} \right)\right]}{\prod\limits_{1\le i<j\le n}(u_{i}-u_{j})},
\end{equation*}
depending on weakly decreasing sequences $\lambda=(\lambda_1, \dots, \lambda_n)$ of nonnegative integers, a parameter $q$ and inhomogeneities $\xi_x$ and $s_x$ for nonnegative integers~$x$.

For $u=v$, the modified Robbins polynomials are---up to an unimportant factor---a special case of the fully inhomogeneous spin Hall--Littlewood symmetric rational functions: if we set $\xi_x=1$ and $s_x=q^{-\half}$,
and map $u_i \mapsto \frac{1 + q^{\half} x_i}{q^{\half} + x_i}$ for all $1 \le i \le n$, we transform $\F_{\lambda}(u_1,\dots,u_n)$ into 
\begin{equation*}
	(-1)^{\binom{n+1}{2}}
	q^{\binom{n+1}{2}/2} \prod_{i=1}^{n} (q^{\half} + x_i) R_{\boldk}^{\ast}(x_{1},\ldots,x_{n};1,1,q^{\half}+q^{-\half})\end{equation*}
where $\boldk$ is the reverse sequence of $\lambda$.

In all cases, however, these generalized Littlewood identities for Hall--Littlewood polynomials and their generalizations are not of the classical type as they either sum over all partitions with even parts or over all partitions whose conjugates have even parts. Although, in the Schur case, the classical Littlewood identity follows from these variants via a Pieri rule, this does not appear to be true for the more general inhomogeneous spin Hall--Littlewood symmetric rational functions, due to the nature of the Pieri rules available for the latter. However, we conjecture in Section~\ref{sec:conjecture} a generalization of the Littlewood-type identity established in this paper for fully inhomogeneous spin Hall--Littlewood symmetric rational functions. This generalization would imply a new Littlewood identity for Hall--Littlewood polynomials of the classical type, see \eqref{eq:new_Hall_Littlewood}. On the other hand, the Cauchy-type identity accompanying our Littlewood-type identity follows from the Cauchy identity provided in \cite{P21} by the specialization given above.

\section{Down-arrowed monotone triangles}
\label{sec:down}

In this section, we introduce down-arrowed monotone triangles (DAMTs), which are a combinatorial model for modified Robbins polynomials. They are defined as decorated Gelfand--Tsetlin patterns.

A \deff{Gelfand--Tsetlin pattern} with $n$ rows is a triangular array of integers
\begin{equation*}
	\begin{array}{ccccccccc}
		&&&& a_{1,1} \\
		&&& a_{2,1} && a_{2,2} \\
		&& \iddots && \iddots\ddots &&\ddots\\
		&a_{n-1,1} && \cdots && \cdots && a_{n-1,n-1} \\
		a_{n,1} && a_{n,2} && \cdots && a_{n,n-1} && a_{n,n}
	\end{array}
\end{equation*}
with weakly increasing entries along $\nearrow$-diagonals and $\searrow$-diagonals.
Gelfand--Tsetlin patterns with $n$ rows are in bijection with semistandard Young tableaux with entries less than or equal to $n$ \cite[Section~7.10]{Sta99} as follows: There are $a_{1,1}$ $1$s in the first row of the Young tableau. Then, for each $i \in \{ 2,\dots, n\}$, the entries equal to $i$ form the skew shape $(a_{i,i},\dots,a_{i,1})/(a_{i-1,i-1},\dots,a_{i-1,1})$.

The \deff{Schur polynomials} $s_{(\lambda_1,\dots,\lambda_m)}(x_1,\ldots,x_n)$ are combinatorially defined as the generating functions of semistandard Young tableaux of shape~$(\lambda_1,\dots,\lambda_m)$ with entries less than or equal to $n$ with respect to the weight $\prod_{i=1}^{n} x_i\textsuperscript{$\# i$'s}$. By assigning the weight
\begin{equation*}
	\prod_{i=1}^n x_i^{\sum_{j=1}^{i} a_{i,j}-\sum_{j=1}^{i-1} a_{i-1,j}}
\end{equation*}
to each Gelfand--Tsetlin pattern, we see that the previous bijection is weight-preserving and that $s_{(\lambda_1,\dots,\lambda_m)}(x_1,\ldots,x_n)$ is also the generating function of Gelfand--Tsetlin patterns with $n$ rows and bottom row $(\lambda_m,\dots,\lambda_1)$, adding leading zeroes if necessary.

\begin{exm}
	We present the following Gelfand--Testlin pattern with four rows and weight~$x_1^3  x_2^3 \allowbreak x_3^4  x_4^3$ and its corresponding semistandard Young tableau:
	
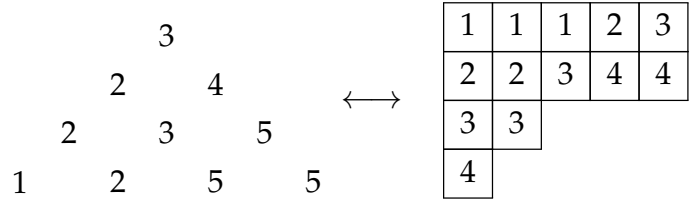
\begin{equation*}
		\begin{tikzpicture}[baseline=(current bounding box.center),scale=0.65]
			\node at (0,0) {$1$};
			\node at (2,0) {$2$};
			\node at (4,0) {$5$};
			\node at (6,0) {$5$};
			\node at (1,1) {$2$};
			\node at (3,1) {$3$};
			\node at (5,1) {$5$};
			\node at (2,2) {$2$};
			\node at (4,2) {$4$};
			\node at (3,3) {$3$};
		\end{tikzpicture}
		\longleftrightarrow\quad
		\ytableaushort{11123,22344,33,4}
	\end{equation*}
\end{exm}

\deff{Monotone triangles} are Gelfand--Tsetlin patterns with strictly increasing rows. It is well known that monotone triangles with bottom row $(1,2,\dots,n)$ are in bijection with ASMs of order~$n$ \cite{MRR83}.

\begin{exm}
	The following matrix on the left is an ASM of order~$4$. It corresponds to the monotone triangle on the right.
	
	We illustrate the bijection with the matrix in the middle: To each entry of an ASM of order~$n$, we add the entries in the same column above. This results in rows of $0$s and $1$s, where the $i$th row comprises exactly $i$ $1$s for all $1 \le i \le n$. Recording the positions of these $1$s for every row finally yields a monotone triangle.
	\begin{equation*}
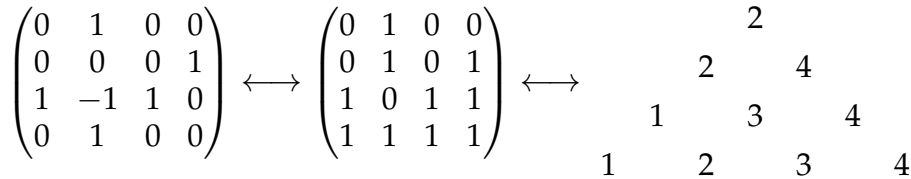
		\begin{pmatrix}
			0 & 1 & 0 & 0\\
			0 & 0 & 0 & 1\\ 
			1 &-1 & 1 & 0\\ 
			0 & 1 & 0 & 0
		\end{pmatrix}
		\longleftrightarrow
		\begin{pmatrix}
			0 & 1 & 0 & 0\\
			0 & 1 & 0 & 1\\ 
			1 & 0 & 1 & 1\\ 
			1 & 1 & 1 & 1
		\end{pmatrix}
		\longleftrightarrow
		\begin{tikzpicture}[baseline=(current bounding box.center),scale=0.65]
			\node at (0,0) {$1$};
			\node at (2,0) {$2$};
			\node at (4,0) {$3$};
			\node at (6,0) {$4$};
			\node at (1,1) {$1$};
			\node at (3,1) {$3$};
			\node at (5,1) {$4$};
			\node at (2,2) {$2$};
			\node at (4,2) {$4$};
			\node at (3,3) {$2$};
			
		\end{tikzpicture}
	\end{equation*}
\end{exm}

In order to obtain a combinatorial interpretation of the modified Robbins polynomials, we introduce down-arrowed monotone triangles (DAMTs).
To this end, we distinguish entries that are not in the bottom row in relation to their neighbouring entries in the row below: An entry that is equal to its $\swarrow$-neighbour is called \deff{left-leaning}; conversely, an entry that is equal to its $\searrow$-neighbour is called \deff{right-leaning}. An entry that is neither left-leaning nor right-leaning is called \deff{special}.

\begin{dfn}
	A \deff{down-arrowed monotone triangle} is a monotone triangle where each  entry not in the bottom row is equipped with one of the decorations $\swarrow$, $\downarrow$ or $\searrow$ according to the following rules:
	\begin{itemize}
		\item Every left-leaning entry is decorated with $\swarrow$.
		\item Every right-leaning entry is decorated with $\searrow$.
	\end{itemize}
	An entry equipped with $\downarrow$ is called \deff{central}. The following weight is assigned to a DAMT:
	\begin{equation}\label{eq:weight_DAMT}
		u^{\# \searrow}  v^{\# \swarrow}  w^{\# \downarrow}
		\prod_{i=1}^n x_i^{\sum_{j=1}^{i} a_{i,j}-\sum_{j=1}^{i-1} a_{i-1,j} + \# \searrow \text{ in row~$(i-1)$ } - \# \swarrow \text{ in row~$(i-1)$}}.
	\end{equation}
\end{dfn}

\begin{exm}
	We provide an example of a DAMT with weight $u^3 v w^2 x_1^3 x_2^5 x_3^3 x_4^5$:
	\begin{equation*}
		\begin{tikzpicture}[baseline=(current bounding box.center),scale=0.65]
			\node at (0,0) {$1$};
			\node at (2,0) {$2$};
			\node at (4,0) {$4$};
			\node at (6,0) {$7$};
			\node at (1,1) {$1$};
			\node at (3,1) {$3$};
			\node at (5,1) {$5$};
			\node at (2,2) {$2$};
			\node at (4,2) {$5$};
			\node at (3,3) {$3$};
			
			\node at (.5,.45) {\footnotesize $\swarrow$};
			\node at (3.5,.45) {\footnotesize $\searrow$};
			\node at (5,.45) {\footnotesize $\downarrow$};
			\node at (2,1.45) {\footnotesize $\downarrow$};
			\node at (4.5,1.45) {\footnotesize $\searrow$};
			\node at (3.5,2.45) {\footnotesize $\searrow$};
		\end{tikzpicture}
	\end{equation*}
\end{exm}

It has been shown in \cite[Theorem~3.1]{FS23} that the generating function of all DAMTs with a prescribed bottom row is given by modified Robbins polynomials.

\begin{thm}
	\label{thm:gfun}
	Let $\boldk=(k_{1},\ldots,k_{n})$ be a strictly increasing sequence of integers. The generating function of DAMTs with bottom row~$\boldk$ with respect to the weight~\eqref{eq:weight_DAMT} is given by the modified Robbins polynomial $R_{\boldk}^{\ast}(x_{1},\ldots,x_{n};u,v,w)$.
\end{thm}

We may also define a weight function on monotone triangles in order to obtain the same generating function $R_{\boldk}^{\ast}(1,\ldots,1;u,v,w)$ (and thereby avoid the necessity of decorations), see \cite{FS23}. In this case, the weight is not monomial anymore; however, it reduces to $1$ when setting $u=v=1$, $w=-1$ and $x_1=\dots=x_n=1$. Consequently, $R_{\boldk}^{\ast}(1,\ldots,1;1,1,-1)$ yields the number of (undecorated) monotone triangles with bottom row~$\boldk$. 

More precisely, let $M$ be a monotone triangle with $n$ rows. We denote the number of left-leaning and right-leaning entries in row~$i$ by $l_i(M)$ and $r_i(M)$, respectively. Analogously, we denote the number of special entries in row~$i$ by $s_i(M)$. The number of all left-leaning and right-leaning entries is denoted by $l(M)$ and $r(M)$, respectively.

We define the following weight on any monotone triangle~$M$:
\begin{equation}\label{eq:weight_MT}
	u^{r(M)}  v^{l(M)}
	\prod_{i=1}^n x_i^{\sum_{j=1}^{i} a_{i,j}-\sum_{j=1}^{i-1} a_{i-1,j} + r_{i-1}(M) - l_{i-1}(M)} \left( w+ u x_i + v x_i^{-1} \right)^{s_{i-1}(M)}.
\end{equation}

We obtain DAMTs from monotone triangles by decorating the entries that are not in the bottom row. While the decorations for left-leaning and right-leaning entries are forced, we have three different choices for every special entry. This observation explains that we obtain the weight~\eqref{eq:weight_DAMT} by expanding \eqref{eq:weight_MT} into monomials. Since \eqref{eq:weight_MT} equals obviously $1$ by setting $x_1=\dots=x_n=1$, $u=v=1$ and $w=-1$, it follows that $R_{\boldk}^{\ast}(1,\ldots,1;1,1,-1)$ yields the number of monotone triangles with bottom row~$\boldk$.

In \cite{FS23}, Schreier-Aigner and the first author also considered \deff{arrowed monotone triangles} and \deff{extended arrowed monotone triangles}; these are further variants of decorated Gelfand--Tsetlin patterns.

An arrowed monotone triangle is a monotone triangle where each entry is decorated with one of the decorations $\nwarrow$, $\nearrow$ or $\nenwarrow$ according to the following rules
\begin{itemize}
	\item If an entry coincides with its $\nwarrow$-neighbour, then it is decorated with $\nearrow$.
	\item If an entry coincides with its $\nearrow$-neighbour, then it is decorated with $\nwarrow$.
\end{itemize}
An extended arrowed monotone triangle is a triangular array of integers that are not necessarily weakly increasing along rows and where the entries are decorated with $\nwarrow$, $\nearrow$, $\nenwarrow$ or $\emptyset$. We omit the rules for the decorations here, see \cite{FS23} for further details as well as further generalizations. Arrowed monotone triangles are a proper subset of extended arrowed monotone triangles: we recover the notion of arrowed monotone triangles by restricting extended arrowed monotone triangles to strictly increasing bottom rows and omitting $\emptyset$ as decoration.

The generating function of (extended) arrowed monotone triangles is given by the ordinary \deff{Robbins polynomials} $R_{\boldk} (x_{1},\ldots,x_{n};t,u,v,w)$. Modified Robbins polynomials and (nonmodified) Robbins polynomials fulfill the following relation:
\begin{equation*}
	R_{\boldk}^{\ast}(x_{1},\ldots,x_{n};u,v,w) = \left. \frac{R_{\boldk}(x_{1},\ldots,x_{n};t,u,v,w)}{\prod_{i=1}^{n} \left(t+u x_i + \frac{v}{x_i}   + w \right)}\right|_{t=0}.
\end{equation*}
Robbins polynomials generalize Hall--Littlewood polynomials~$P$. In fact, let $\lambda$ be an integer partition and denote the multiplicity of an entry~$i$ by $m_i(\lambda)$, then
\begin{equation*}
	P_{\lambda}(x_1,\dots,x_n;q) = (-1)^{\binom{n}{2}} \prod_{r \ge 0}  \frac{1}{(q;q)_{m_r(\lambda)}} R_{\lambda}(x_{1},\ldots,x_{n};-q,0,0,1),
\end{equation*} 
where
\begin{equation*}
	(a;q)_k \coloneq \prod_{i=0}^{k-1} (1-a q^i)
\end{equation*}
denotes the standard \deff{$q$-Pochhammer symbol}.

The feature of Robbins polynomials is that they admit a combinatorial interpretation for arbitrary integer sequences~$\boldk$ through extended arrowed monotone triangles. However, this combinatorial interpretation involves signs in contrast to arrowed monotone triangles.

\section{Proof of Theorem~\ref{thm:Robbins_Littlewood}}
\label{sec:mainproof}

Littlewood-type identities are often proved by what is occasionally referred to as the \emph{Izergin--Korepin technique}: identifying a set of properties that uniquely determines one side of the equation and subsequently showing that the other side satisfies these properties as well.

We prove Theorem~\ref{thm:Robbins_Littlewood} in a different and elementary way:  The antisymmetrizer in the definition of modified Robbins polynomials enables us to establish a recursion relation for these polynomials. We then show directly that the right-hand side of \eqref{eq:Robbins_Littlewood} fulfills the same recursion as the left hand side. In addition, a sketch of a second proof using the Izergin--Korepin technique is provided as well. At the end of this section, we show how to derive Corollary~\ref{cor:Littlewood}.

The Littlewood-type identity for modified Robbins polynomials  involves the notion of Pfaffians which we recall in the following before proving Theorem~\ref{thm:Robbins_Littlewood}:

For every triangular array $A=(a_{i,j})_{1 \le i < j \le 2n}$, the Pfaffian $\pf(A)$ is defined as
\begin{equation*}
	\sum_{\{(i_1,j_1),\dots,(i_n,j_n)\}} \sgn(i_1 j_1 \dots i_n j_n) a_{i_1,j_1} \cdots a_{i_n,j_n},
\end{equation*}
where we sum over all perfect matchings $\{(i_1,j_1),\dots,(i_n,j_n)\}$ of $\{1,\dots,2n\}$ such that $i_1 < \dots < i_n$ and $i_k < j_k$ for all $1 \le k \le n$.

Every triangular array $A=(a_{i,j})_{1 \le i < j \le 2n}$ is the upper triangular part of a uniquely determined $2n \times 2n$ skew-symmetric matrix $M$, which we call the \emph{skew-symmetric completion} of $A$. A matrix $M=(m_{i,j})_{1 \le i , j \le 2n}$ is skew-symmetric if $m_{i,j} = - m_{j,i}$ for all $1 \le i,j \le 2n$; in particular, all diagonal entries are $0$. It is well known that $\pf(A)^2=\det(M)$. On that account, we will repeatedly change perspectives when dealing with Pfaffians as we will consider either the triangular array or the skew-symmetric matrix that underlies the Pfaffian.

We recall the main methods for manipulating Pfaffians used in the present paper:
\begin{itemize}
	\item Regarding elementary row or column operations on the underlying  matrix, we perform corresponding row and column operations at the same time. Thus, the resulting matrix remains skew-symmetric. In particular, simultaneously swapping corresponding rows and columns multiplies the Pfaffian by $-1$.
	\item Expanding the Pfaffian along the last column yields
	\begin{equation*}
		\pf_{1 \le i < j \le 2n} (a_{i,j}) = \sum_{k=1}^{2n-1} (-1)^{k+1} a_{k,2n} \pf_{\myatop{1 \le i < j \le 2n}{i,j \neq k}} (a_{i,j}).
	\end{equation*}
	\item For all $c_1,\dots,c_{2n}$, it holds that
	\begin{equation}\label{eq:linearity_pfaffian}
		\pf_{1 \le i < j \le 2n} (c_i c_j a_{i,j}) = \prod_{i=1}^{2n} c_i \pf_{1 \le i < j \le 2n} (a_{i,j}).
	\end{equation}
	Thus, for a given~$i$, setting $c_i=c$ and all other $c_1,\dots,c_{i-1},c_{i+1},\dots,c_{2n}$ equal to $1$ corresponds to multiplying the $i$th row and the $i$th column of the underlying matrix by $c$.
\end{itemize}

In preparation for the proof of Theorem~\ref{thm:Robbins_Littlewood}, we show the following lemma which will be used for certain row and column operations:

\begin{lem}\label{lem:MatrixOperationEven}
	The following identities hold:
	\begin{enumerate}[label=(\roman*)]
		\item\label{eq:MatrixOperationEven2} For each positive integer~$n$, it holds
		\begin{equation*}
			\sum_{j=1}^{n} \left(\frac{1+x_j}{x_j} \prod_{\myatop{1 \le p \le n}{p \neq j}} \frac{1-x_j x_p}{x_j -x_p} \right) + (-1)^n \prod_{k=1}^{n} \frac{1}{x_k} = 1.
		\end{equation*}
		\item\label{eq:MatrixOperationEven1} For each positive integer~$n$ and $i$ satisfying $1 \le i \le n$, it holds
		\begin{multline*}
			(-1)^n \sum_{j=1}^{n} \left( \frac{(x_j-x_i)(1+(1+w) x_i x_j)}{(x_i+x_j+w x_i x_j)(1-x_i x_j)} \cdot \frac{1+x_j}{x_j} \prod_{\myatop{1 \le p \le n}{p \neq j}} \frac{1-x_j x_p}{x_j -x_p} \right)\\
			+ \frac{1-x_i}{x_i} \prod\limits_{\myatop{1 \le p \le n}{p \neq i}} \frac{1+ w x_i + x_i x_p}{x_i + x_p + w x_i x_p} = (-1)^n+\prod_{k=1}^{n} \frac{1}{x_k}.
		\end{multline*}
	\end{enumerate}
\end{lem}

\begin{proof}
	Let us first prove \ref{eq:MatrixOperationEven2}. After some straightforward manipulation, the expression in \ref{eq:MatrixOperationEven2} becomes the following polynomial identity:
	\begin{multline}\label{eq:MatrixOperationEven2_1}
		\sum_{j=1}^{n} \left((-1)^j (1+x_j) \prod_{\myatop{1 \le p \le n}{p \neq j}} x_p(1-x_j x_p) \prod_{\myatop{1 \le p < q\le n}{p,q \neq j}} (x_q - x_p) \right)\\
		=  \left( -1 + (-1)^n \prod_{p=1}^{n} x_p \right)  \prod_{1 \le p < q \le n} (x_q - x_p).
	\end{multline}
	Both sides are multivariate polynomials in $x_1,\dots,x_n$ with highest degree at most $n$ in each variable. It is not difficult to see that both sides vanish if we set $x_s=x_t$ for any $s \neq t$. Furthermore, if $x_s=0$, both sides evaluate to 
	\begin{equation*}
		(-1)^s  \prod_{\myatop{1 \le p \le n}{p \neq s}} x_p \prod_{\myatop{1 \le p < q \le n}{p,q \neq s}} (x_q - x_p),
	\end{equation*}
	resulting in $n$ coinciding evaluations for each variable. Thus, both sides are equal up to a multiplicative factor independent of $x_1,\dots,x_n$. Since the coefficient of $x_n^n x_{n-1}^{n-1} \cdots x_1$, that is, the leading coefficient in the lexicographic order for $x_n > \dots > x_1$, equals $(-1)^n$ on both sides, we have shown \eqref{eq:MatrixOperationEven2_1}.
	
	Regarding \ref{eq:MatrixOperationEven1}, we transform the identity into
	\begin{multline}\label{eq:MatrixOperationEven1_1}
		\sum_{j=1}^{n} \left( \vphantom{ \prod_{\myatop{1 \le p < q \le n}{p,q \neq j}}}  (-1)^{j} (1+(1+w) x_i x_j) (1+x_j) (x_j-x_i) \right.\\
		\left. \times \prod_{\myatop{1 \le p \le n}{p \neq j}} x_p \prod_{\myatop{1 \le p \le n}{p \neq i,j}} (1-x_j x_p)  (x_i+x_p+w x_i x_p) \prod_{\myatop{1 \le p < q \le n}{p,q \neq j}} (x_q - x_p) \right) \\
		+ (1-x_i)  \prod\limits_{\myatop{1 \le p \le n}{p \neq i}} x_p (1+ w x_i + x_i x_p) \prod_{1 \le p <q \le n} (x_q -x_p)\\
		= \left(1+ (-1)^n \prod_{k=1}^{n} x_k \right) \prod_{\myatop{1 \le p \le n}{p \neq i}} (x_i+x_p+w x_i x_p) \prod_{1 \le p < q \le n } (x_q -x_p),
	\end{multline}
	both sides of which are multivariate polynomials in $x_1,\dots,x_n$; the highest degree of $x_s$ is at most $n+1$ if $s \neq i$ and $2n-1$ otherwise.
	
	A straightforward computation shows that both sides vanish if we set $x_s=x_t$ for any $s \neq t$. Setting $x_s=0$ yields
	\begin{equation*}
		(-1)^{s-1} x_i^2 \prod_{\myatop{1 \le p \le n}{p \neq i,s}} x_p (x_i+x_p+w x_i x_p) \prod_{\myatop{1 \le p < q \le n}{p,q \neq s}} (x_q - x_p)
	\end{equation*}
	on both sides if $s \neq i$; otherwise, we obtain
	\begin{equation*}
		(-1)^{i-1}  \prod_{\myatop{1 \le p \le n}{p \neq i}} x_p^2 \prod_{\myatop{1 \le p < q \le n}{p,q \neq i}} (x_q - x_p).
	\end{equation*}
	
	Next, we consider the evaluation for $x_s=-\frac{x_i}{1+w x_i}$ if $s \neq i$. The right-hand side of \eqref{eq:MatrixOperationEven1_1} clearly vanishes. The left-hand side, however, becomes
	\begin{multline*}
		\frac{(-1)^{s}}{(1+w x_i)^{n+1}} x_i (x_i-1)(2+ w x_i)(1-x_i+w x_i)(1+x_i+w x_i)\\
		\times \prod_{\myatop{1 \le p \le n}{p \neq s}} x_p \prod_{\myatop{1 \le p \le n}{p \neq i,s}} (1+w x_i+ x_i x_p)  (x_i+x_p+w x_i x_p) \prod_{\myatop{1 \le p < q \le n}{p,q \neq s}} (x_q - x_p)\\
		- \frac{(-1)^s}{(1+w x_i)^{n+1}} x_i (x_i-1) (1-x_i+w x_i)(1+x_i + w x_i)\\
		\times\prod_{\myatop{1 \le p \le n}{p \neq s}}  (x_i+x_p + w x_i x_p)  \prod_{\myatop{1 \le p \le n}{p \neq i,s}} x_p (1+ w x_i + x_i x_p) \prod_{\myatop{1 \le p <q \le n}{p,q \neq s}} (x_q -x_p).
	\end{multline*}
	After omitting common factors, we are left with $(2+ w x_i) x_i -  (x_i+x_i + w x_i^2)$, which equals zero. 
	
	Next, let $x_i=-\frac{x_s}{1+w x_s}$ for any $s \neq i$. It is again obvious that the right-hand side of \eqref{eq:MatrixOperationEven1_1} vanishes. We obtain for the left-hand side
	\begin{multline*}
		\frac{(-1)^{s+i-\left[s<i\right]}}{(1+w x_s)^{2n-1}} x_s^2 (1-x_s) (1+x_s) (2+ w x_s)(1+x_s+w x_s)\\
		\times  \prod_{\myatop{1 \le p \le n}{p \neq i,s}}x_p (1-x_s x_p)  (x_p - x_s)  (x_s+x_p+w x_s x_p) \prod_{\myatop{1 \le p < q \le n}{p,q \neq i,s}} (x_q - x_p)\\
		+ \frac{(-1)^{i-1}}{(1+w x_s)^{2n-1}}  (1+x_s + w x_s)\prod_{\myatop{1 \le p \le n}{p \neq i}} x_p (1-x_s x_p) (x_s+x_p+w x_s x_p) \\
		\times \prod_{\myatop{1 \le p <q \le n}{p,q \neq i}} (x_q -x_p),
	\end{multline*}
	where $\left[\boldsymbol{\cdot}\right]$ denotes the \emph{Iverson bracket}: For any proposition~$P$, $\left[P\right]$ evaluates to $1$ if $P$ is true and to $0$ if $P$ is false. It is again not difficult to see that the previous expression is zero.
	
	In total, we have now $n+1$ matching evaluations for $x_s$ for any $s\neq i$ and $2n-1$ evaluations for $x_i$. As a consequence, we know that the difference between the left-hand side and the right-hand side of \eqref{eq:MatrixOperationEven1_1} is the product of
	\begin{equation}\label{eq:differencepoly}
		\prod_{p=1}^{n} x_p \prod_{1 \le p < q \le n} (x_q-x_p) \prod_{\myatop{1 \le p \le n}{p \neq i}} (x_i + x_p + w x_i x_p)^2
	\end{equation}
	and a rational expression in $w$. In fact, we want to show that this factor equals zero. To this end, we consider the specializations $(x_1,\dots,x_n)=(x_1,\frac{1}{x_1},x_3,\frac{1}{x_3},\dots,x_{n-1},\frac{1}{x_{n-1}})$ if $n$ is even and $(x_1,\dots,x_n)=(x_1,\frac{1}{x_1},x_3,\frac{1}{x_3},\dots,x_{n-2},\frac{1}{x_{n-2}},-1)$ if $n$ is odd for which \eqref{eq:differencepoly} does not vanish. Given these specializations, we will show that the left-hand side of \eqref{eq:MatrixOperationEven1_1} coincides with the right-hand side, which completes the proof of \eqref{eq:MatrixOperationEven1_1}.
	
	First, we consider the left-hand side of \eqref{eq:MatrixOperationEven1_1}. Due to the factor $\prod_{\myatop{1 \le p \le n}{p \neq i,j}} (1-x_j x_p)$, we see that the first sum vanishes unless $j=i+1$ if $i$ is odd and $j=i-1$ if $i$ is even. Let us assume the latter. Thus, it suffices to only keep the $(i-1)$th summand in the first sum. After some manipulation, the left-hand side of \eqref{eq:MatrixOperationEven1_1} reduces to
	\begin{multline*}
		x_i (1+(1+w) x_{i-1} x_i) (1+x_{i-1}) (x_i-x_{i-1})\\
		\times \prod_{\myatop{1 \le p \le n}{p \neq i-1,i}} x_p (1-x_{i-1} x_p)  (x_i+x_p+w x_i x_p) (x_p - x_i) \prod_{\myatop{1 \le p < q \le n}{p,q \neq i-1,i}} (x_q - x_p) \\
		+ x_{i-1} (1-x_i) (1+w x_i+x_{i-1} x_i) (x_i-x_{i-1})\\
		\times\prod\limits_{\myatop{1 \le p \le n}{p \neq i-1,i}} x_p (1+ w x_i + x_i x_p) (x_p-x_i) (x_p-x_{i-1}) \prod_{\myatop{1 \le p <q \le n}{p,q \neq i-1,i}} (x_q -x_p).
	\end{multline*}
	
	The right-hand side of \eqref{eq:MatrixOperationEven1_1}, however, can be recasted as
	\begin{multline*}
		(x_{i-1}+x_i+w x_{i-1} x_i) (x_i-x_{i-1})\\
		\times \left(1+ (-1)^n  \prod_{k=1}^{n} x_k \right)\prod\limits_{\myatop{1 \le p \le n}{p \neq i-1,i}} (x_i+ x_p +w  x_i x_p) (x_p-x_i) (x_p-x_{i-1}) \prod_{\myatop{1 \le p <q \le n}{p,q \neq i-1,i}} (x_q -x_p).
	\end{multline*}
	After cancelling out common factors in the two previous expressions, it remains to show that
	\begin{multline*}
		x_i (1+(1+w) x_{i-1} x_i) (1+x_{i-1}) \prod_{\myatop{1 \le p \le n}{p \neq i-1,i}} x_p (1-x_{i-1} x_p)  (x_i+x_p+w x_i x_p) \\
		+ x_{i-1} (1-x_i) (1+w x_i+x_{i-1} x_i) \prod\limits_{\myatop{1 \le p \le n}{p \neq i-1,i}} x_p (1+ w x_i + x_i x_p) (x_p-x_{i-1}) 
	\end{multline*}
	equals
	\begin{equation*}
		(x_{i-1}+x_i+w x_{i-1} x_i) \left(1+ (-1)^n \prod_{k=1}^{n} x_k \right)\prod\limits_{\myatop{1 \le p \le n}{p \neq i-1,i}} (x_i+ x_p +w  x_i x_p) (x_p-x_{i-1}) 
	\end{equation*}
	for the specializations under consideration. For this purpose, let $k \neq i$ be even and consider the products
	\begin{itemize}
		\item $x_{k-1} (1-x_{i-1} x_{k-1})  (x_i+x_{k-1}+w x_i x_{k-1}) \cdot x_{k} (1-x_{i-1} x_{k})  (x_i+x_{k}+w x_i x_{k})$,
		\item $x_{k-1} (1+ w x_i + x_i x_{k-1}) (x_{k-1}-x_{i-1}) \cdot x_{k} (1+ w x_i + x_i x_{k}) (x_{k}-x_{i-1})$,
		\item $(x_i+ x_{k-1} +w  x_i x_{k-1}) (x_{k-1}-x_{i-1}) \cdot (x_i+ x_{k} +w  x_i x_{k}) (x_{k}-x_{i-1})$,
	\end{itemize}
	all of which evaluate to
	\begin{equation*}
		\frac{1}{x_{k-1}^2 x_i^2} (x_{k-1}-x_{i-1})(1-x_{i-1} x_{k-1})(w+x_{i-1}+x_{k-1})(1+w x_{k-1}+x_{i-1} x_{k-1})
	\end{equation*}
	when setting $x_i \mapsto \frac{1}{x_{i-1}}$ and $x_k \mapsto \frac{1}{x_{k-1}}$. In addition, in the case of $n$ being odd, one can easily check that all of the following expressions
	\begin{itemize}
		\item $x_{n} (1-x_{i-1} x_{n})  (x_i+x_{n}+w x_i x_{n})$,
		\item $x_{n} (1+ w x_i + x_i x_{n}) (x_{n}-x_{i-1})$,
		\item $(x_i+ x_{n} +w  x_i x_{n}) (x_{n}-x_{i-1})$
	\end{itemize}
	are equal to $(1+x_{i-1})(1 - x_i + w x_i)$ if setting $x_n=-1$.
	
	Moreover, we see that
	\begin{equation*}
		x_i (1+(1+w) x_{i-1} x_i) (1+x_{i-1}) + x_{i-1} (1-x_i) (1+w x_i+x_{i-1} x_i)  
	\end{equation*}
	and
	\begin{equation*}
		(x_{i-1}+x_i+w x_{i-1} x_i) \left(1+ (-1)^n \prod_{k=1}^{n} x_k \right)
	\end{equation*}
	both specialize to
	\begin{equation*}
		\frac{2}{x_{i-1}} (1+ w x_{i-1}+ x_{i-1}^2).
	\end{equation*}
	This finally concludes the proof of \eqref{eq:MatrixOperationEven1_1} for even $i$. An analogous argument holds in the case of $i$ being odd.
\end{proof}

We are now ready to proceed with the main proof of Theorem~\ref{thm:Robbins_Littlewood}.

\begin{proof}[First proof of Theorem~\ref{thm:Robbins_Littlewood}]
	We prove the statement by induction on $n$, showing that both sides of \eqref{eq:Robbins_Littlewood} satisfy the same recurrence relation.
	
	To start with, recall that $R_{\boldk}^{\ast}(x_{1},\ldots,x_{n};1,1,w)$ is given by
	\begin{equation*}
		\frac{\asym_{x_{1},\ldots,x_{n}}\left[\prod_{1\le i<j\le n}(x_{i}x_{j}+1+wx_{i})\prod_{i=1}^{n}x_{i}^{k_{i}}\right]}{\prod_{1\le i<j\le n}(x_{j}-x_{i})}.
	\end{equation*}
	Through geometric series evaluation, it follows that
	\begin{equation*}
		F(x_1,\ldots,x_n) \coloneqq \sum_{0 \le k_1 < \dots < k_n} R^{\ast}_{(k_{1},\ldots,k_{n})}(x_{1},\ldots,x_{n};1,1,w)
	\end{equation*} equals
	\begin{equation}\label{eq:asym_expression} 
		\frac{\asym_{x_{1},\ldots,x_{n}}\left[\prod_{1\le i<j\le n}(x_{i}x_{j}+1+wx_{i})\prod_{i=1}^{n}x_{i}^{i-1}(1-\prod_{j=i}^n x_j)^{-1}\right]}{\prod_{1\le i<j\le n}(x_{j}-x_{i})} 
	\end{equation}
	since 	
	\begin{multline*}
		\sum_{0 \le k_1 < \dots < k_n} \prod_{i=1}^{n} x_i^{k_i} =\prod_{i=1}^{n} x_i^{i-1} \sum_{0 \le k_1 \le \dots \le k_n} \prod_{i=1}^{n} x_i^{k_i} \\
		= \prod_{i=1}^{n} x_i^{i-1} \sum_{0 \le k_1 \le \dots \le k_n} (x_1 \dots x_n)^{k_1} (x_2 \dots x_n)^{k_2-k_1} \cdots x_n^{k_n-k_{n-1}}\\
		= \prod_{i=1}^{n} x_i^{i-1} \sum_{k_1=0}^{\infty} (x_1 \dots x_n)^{k_1} \sum_{k_2=0}^{\infty} (x_2 \dots x_n)^{k_2} \cdots  \sum_{k_n=0}^{\infty} x_n^{k_n} = \prod_{i=1}^{n} \frac{x_i^{i-1}}{1-\prod_{j=i}^{n} x_j}.
	\end{multline*}
	Thus, we obtain the recurrence relation
	\begin{equation}\label{eq:rec_Littlewood}
		F(x_1,\ldots,x_n) = 
		\left(1-\prod_{i=1}^n x_i \right)^{-1}  \sum\limits_{k=1}^n \prod\limits_{\myatop{1 \le i \le n}{i \neq k}} \frac{(x_i x_k+1+w x_k) x_i}{x_i-x_k}
		F(x_1,\ldots,\widehat{x_k},\ldots,x_n)
	\end{equation}
	with initial condition
	\begin{equation*}
		F(x_1)=\frac{1}{1-x_1},
	\end{equation*}
	where $x_1,\ldots,\widehat{x_k},\ldots,x_n$ denotes the omission of $x_k$ in $x_1,\ldots,x_n$. 
	
	\NiceMatrixOptions{cell-space-limits=5pt}
	
	If $n=1$, then both sides of \eqref{eq:Robbins_Littlewood} clearly evaluate to $1/(1-x_1)$. Now consider the induction step from $n-1$ to $n$. We make a case distinction regarding the parity of $n$.
	
	Let us first assume that $n$ is even. Using the induction hypothesis and the recursion~\eqref{eq:rec_Littlewood}, it follows that the right-hand side of \eqref{eq:Robbins_Littlewood} equals
	\begin{multline*}
		\left(1-\prod_{i=1}^n x_i\right)^{-1} \sum\limits_{k=1}^n \prod\limits_{\myatop{1 \le i \le n}{i \neq k}} \frac{(x_i x_k+w x_k + 1) x_i}{(x_i-x_k)(1-x_i)}\\
		\times \prod_{\myatop{1 \le i < j \le n}{i,j \neq k}} \frac{x_i+x_j+w x_i x_j}{x_j-x_i} 
		\left. \pf_{\myatop{1 \le i < j \le n+1}{i,j \neq k}} \left( \frac{(x_j-x_i)(1+(1+w) x_i x_j)}{(x_i+x_j+w x_i x_j)(1-x_i x_j)} \right)\right|_{x_{n+1}=1},
	\end{multline*}
	which we rewrite as
	\begin{multline*}
		\left(1-\prod_{i=1}^n x_i\right)^{-1} \prod\limits_{1 \le i < j \le n} \frac{x_i + x_j + w x_i x_j}{x_j-x_i} \sum\limits_{k=1}^n (-1)^{k+1} \left( \prod\limits_{\myatop{1 \le i \le n}{i \neq k}} \frac{1+ w x_k + x_k x_i}{x_k + x_i + w x_k x_i} \right)\\
		\times \prod\limits_{\myatop{1 \le i \le n}{i \neq k}} \frac{x_i}{1-x_i} 
		\left. \pf_{\myatop{1 \le i < j \le n+1}{i,j \neq k}} \left( \frac{(x_j-x_i)(1+(1+w) x_i x_j)}{(x_i+x_j+w x_i x_j)(1-x_i x_j)} \right)\right|_{x_{n+1}=1}.
	\end{multline*}
	Interpreting the sum as an expansion along the last column yields	
	\begin{multline*}
		\left(1-\prod_{i=1}^n x_i\right)^{-1} \prod\limits_{1 \le i < j \le n} \frac{x_i + x_j + w x_i x_j}{x_j-x_i}\\
		\times \pf_{1 \le i < j \le n+2}
		{\scalebox{\myscale}{$\begin{pmatrix}
					\begin{NiceArray}{@{}c|c|c@{}}[corners]
						\left( \frac{x_i}{1-x_i}  \frac{x_j}{1-x_j} \frac{(x_j-x_i)(1+(1+w) x_i x_j)}{(x_i+x_j+w x_i x_j)(1-x_i x_j)} \right)_{1 \le i < j \le n} & \Block{2-1}{\left( \frac{x_i}{1-x_i} \right)_{1 \le i \le n}} & \Block{2-1}{\left( \prod\limits_{\myatop{1 \le p \le n}{p \neq i}} \frac{1+ w x_i + x_i x_p}{x_i + x_p + w x_i x_p} \right)_{1\le i \le n}} \\ \cline{1}
						& & \\ \cline{2-3}
						& & 0
					\end{NiceArray}
				\end{pmatrix}$}}.	
	\end{multline*}
	We extract $\prod_{i=1}^{n} \frac{x_i}{1-x_i}$ from the Pfaffian using~\eqref{eq:linearity_pfaffian} and split it up into two summands as follows:	
	\begin{multline*}
		\pf_{1 \le i < j \le n+2} {\scalebox{\myscale}{$\begin{pmatrix}
					\begin{NiceArray}{@{}c|c|c@{}}[corners]
						\left( \frac{(x_j-x_i)(1+(1+w) x_i x_j)}{(x_i+x_j+w x_i x_j)(1-x_i x_j)} \right)_{1 \le i < j \le n} & \Block{2-1}{\left( 1 \right)_{1 \le i \le n}} & \Block{2-1}{\left( \frac{1-x_i}{x_i}  \left(\prod\limits_{\myatop{1 \le p \le n}{p \neq i}} \frac{1+ w x_i + x_i x_p}{x_i + x_p + w x_i x_p} \right) \right)_{1 \le i \le n}} \\ \cline{1}
						& & \\ \cline{2-3}
						& & -\prod_{k=1}^{n} \frac{1}{x_k}
					\end{NiceArray}
				\end{pmatrix}$}}\\
		+\prod\limits_{i=1}^{n} \frac{1}{x_i} \pf_{1 \le i < j \le n} \left(  \frac{(x_j-x_i)(1+(1+w) x_i x_j)}{(x_i+x_j+w x_i x_j)(1-x_i x_j)} \right).
	\end{multline*}
	This implies that the proof of the even case is complete if we show that	
	\begin{multline}\label{eq:2Pfaffians}
		\pf_{1 \le i < j \le n+2}  {\scalebox{\myscale}{$\begin{pmatrix}
					\begin{NiceArray}{@{}c|c|c@{}}[corners]
						\left( \frac{(x_j-x_i)(1+(1+w) x_i x_j)}{(x_i+x_j+w x_i x_j)(1-x_i x_j)} \right)_{1 \le i < j \le n} & \Block{2-1}{\left( -1 \right)_{1 \le i \le n}} & \Block{2-1}{\left( \frac{1-x_i}{x_i} \left( \prod\limits_{\myatop{1 \le p \le n}{p \neq i}} \frac{1+ w x_i + x_i x_p}{x_i + x_p + w x_i x_p} \right) \right)_{1 \le i \le n}} \\ \cline{1}
						& & \\ \cline{2-3}
						& & \prod_{k=1}^{n} \frac{1}{x_k}
					\end{NiceArray}
				\end{pmatrix}$}}\\
		= \pf_{1 \le i < j \le n} \left(  \frac{(x_j-x_i)(1+(1+w) x_i x_j)}{(x_i+x_j+w x_i x_j)(1-x_i x_j)} \right),
	\end{multline}
	where we incorporated a negative sign into the Pfaffian on the left-hand side by multiplying the $(n+1)$th row and $(n+1)$th column by $-1$ using \eqref{eq:linearity_pfaffian}.
	
	We perform the following elementary row and column operations on the matrix underlying the Pfaffian in the left-hand side of \eqref{eq:2Pfaffians}: For every $j\in\{1,\dots,n\}$, we add the $j$th column times
	\begin{equation*}
		\frac{1+x_j}{x_j} \prod_{\myatop{1 \le p \le n}{p \neq j}} \frac{1-x_j x_p}{x_j -x_p}
	\end{equation*}
	to the last column; we then proceed analogously for the rows. Due to Lemma~\ref{lem:MatrixOperationEven}, this results in the left-hand side of \eqref{eq:2Pfaffians} being equal to
	\begin{equation*}\label{eq:2Pfaffians_1}
		\pf_{1 \le i < j \le n+2}   {\scalebox{\myscale}{$\begin{pmatrix}
					\begin{NiceArray}{@{}c|c|c@{}}[corners]
						\left( \frac{(x_j-x_i)(1+(1+w) x_i x_j)}{(x_i+x_j+w x_i x_j)(1-x_i x_j)} \right)_{1 \le i < j \le n} & \Block{2-1}{\left( -1 \right)_{1 \le i \le n}} & \Block{2-1}{\left(1+\prod_{k=1}^{n} \frac{1}{x_k} \right)_{1 \le i \le n}} \\ \cline{1}
						& & \\ \cline{2-3}
						& & 1
					\end{NiceArray}
				\end{pmatrix}$}}.
	\end{equation*}
	We further manipulate the underlying matrix as follows: We multiply the $(n+1)$th column by $1+\prod_{i=1}^{n} \frac{1}{x_i}$ and add it to the $(n+2)$th column; then we proceed similarly with the last two rows. Thus, we obtain
	\begin{equation*}
		\pf_{1 \le i < j \le n+2}  {\scalebox{\myscale}{$\begin{pmatrix}
					\begin{NiceArray}{@{}c|c|c@{}}[corners]
						\left( \frac{(x_j-x_i)(1+(1+w) x_i x_j)}{(x_i+x_j+w x_i x_j)(1-x_i x_j)} \right)_{1 \le i < j \le n} & \Block{2-1}{\left( -1 \right)_{1 \le i \le n}} & \Block{2-1}{\left(0 \right)_{1 \le i \le n}} \\ \cline{1}
						& & \\ \cline{2-3}
						& & 1
					\end{NiceArray}
				\end{pmatrix}$}}.
	\end{equation*}
	Evaluating this Pfaffian along the last column, finally yields the right-hand side of \eqref{eq:2Pfaffians}.
	
	On the other hand, let us assume that $n$ is odd. Similar as before, applying the recursion~\eqref{eq:rec_Littlewood} to the induction hypothesis and some manipulations imply the right-hand side of \eqref{eq:Robbins_Littlewood} being equal to
	\begin{multline*}
		\left(1-\prod_{i=1}^n x_i \right)^{-1} \prod\limits_{1 \le i < j \le n} \frac{x_i + x_j + w x_i x_j}{x_j-x_i} \sum\limits_{k=1}^n (-1)^{k+1} \left( \prod\limits_{\myatop{1 \le i \le n}{i \neq k}} \frac{1+ w x_k + x_k x_i}{x_k + x_i + w x_k x_i} \right)\\
		\times \prod\limits_{\myatop{1 \le i \le n}{i \neq k}} \frac{x_i}{1-x_i} 
		\pf_{\myatop{1 \le i < j \le n}{i,j \neq k}} \left( \frac{(x_j-x_i)(1+(1+w) x_i x_j)}{(x_i+x_j+w x_i x_j)(1-x_i x_j)} \right).
	\end{multline*}
	We multiply the factor $\prod_{\myatop{1 \le i \le n}{i \neq k}} \frac{x_i}{1-x_i} $ into the Pfaffian using \eqref{eq:linearity_pfaffian} and then interpret the sum as an expansion along the last column, which yields
	\begin{multline*}
		\left(1-\prod_{i=1}^n x_i \right)^{-1} \prod\limits_{1 \le i < j \le n} \frac{x_i + x_j + w x_i x_j}{x_j-x_i}\\
		\times \pf_{1 \le i < j \le n+1}  {\scalebox{\myscale}{$\begin{pmatrix}
					\begin{NiceArray}{@{}c|c@{}}[corners]
						\left( \frac{x_i}{1-x_i}  \frac{x_j}{1-x_j} \frac{(x_j-x_i)(1+(1+w) x_i x_j)}{(x_i+x_j+w x_i x_j)(1-x_i x_j)} \right)_{1 \le i < j \le n} & \Block{2-1}{\left( \prod\limits_{\myatop{1 \le p \le n}{p \neq i}} \frac{1+ w x_i + x_i x_p}{x_i + x_p + w x_i x_p} \right)_{1 \le i \le n}} \\ \cline{1}
						&
					\end{NiceArray}
				\end{pmatrix}$}}.
	\end{multline*}
	Next, we extract again  the product $\prod_{i=1}^{n} \frac{x_i}{1-x_i}$ from the Pfaffian using~\eqref{eq:linearity_pfaffian}
	and recast it by enlarging the dimension of the underlying array as	
	\begin{equation*}
		\pf\limits_{1 \le i < j \le n+3}{\scalebox{\myscale}{$\begin{pmatrix}
					\begin{NiceArray}{@{}c|c|c|c@{}}[corners]
						\left( \frac{(x_j-x_i)(1+(1+w) x_i x_j)}{(x_i+x_j+w x_i x_j)(1-x_i x_j)} \right)_{1 \le i < j \le n} & \Block{2-1}{\left( 1 \right)_{1 \le i \le n}} & \Block{2-1}{\left( \frac{1-x_i}{x_i} \left( \prod\limits_{\myatop{1 \le p \le n}{p \neq i}} \frac{1+ w x_i + x_i x_p}{x_i + x_p + w x_i x_p} \right) \right)_{1 \le i \le n}}  & \Block{2-1}{(0)_{1 \le i \le n}} \\ \cline{1}
						& & & \\ \cline{2-4}
						& & 0 & -1 \\ \cline{3-4}
						& & & 0
					\end{NiceArray}
				\end{pmatrix}$}};
	\end{equation*}
	we then split up the Pfaffian into the following sum:	
	\begin{multline*}
		\pf_{1 \le i < j \le n+3}{\scalebox{\myscale}{$\begin{pmatrix}
					\begin{NiceArray}{@{}c|c|c|c@{}}[corners]
						\left( \frac{(x_j-x_i)(1+(1+w) x_i x_j)}{(x_i+x_j+w x_i x_j)(1-x_i x_j)} \right)_{1 \le i < j \le n} & \Block{2-1}{\left( 1 \right)_{1 \le i \le n}} & \Block{2-1}{ \left( \frac{1-x_i}{x_i} \left( \prod\limits_{\myatop{1 \le p \le n}{p \neq i}} \frac{1+ w x_i + x_i x_p}{x_i + x_p + w x_i x_p} \right) \right)_{1 \le i \le n}}  & \Block{2-1}{(0)_{1 \le i \le n}} \\ \cline{1}
						& & & \\ \cline{2-4}
						& & 0 & -1 \\ \cline{3-4}
						& & & -\prod\limits_{k=1}^{n} \frac{1}{x_k}
					\end{NiceArray}
				\end{pmatrix}$}}\\
		+\prod\limits_{i=1}^{n} \frac{1}{x_i} \pf_{1 \le i < j \le n+1}  {\scalebox{\myscale}{$\begin{pmatrix}
					\begin{NiceArray}{@{}c|c@{}}[corners]
						\left( \frac{(x_j-x_i)(1+(1+w) x_i x_j)}{(x_i+x_j+w x_i x_j)(1-x_i x_j)} \right)_{1 \le i < j \le n} & \Block{2-1}{\left(1 \right)_{1 \le i \le n}} \\ \cline{1}
						&
					\end{NiceArray}
				\end{pmatrix}$}}.
	\end{multline*}
	After putting all together and some minor manipulations, it remains to show that	
	\begin{multline}\label{eq:2Pfaffians_Odd}\\
		\pf_{1 \le i < j \le n+3}{\scalebox{\myscale}{$\begin{pmatrix}
					\begin{NiceArray}{@{}c|c|c|c@{}}[corners]
						\left( \frac{(x_j-x_i)(1+(1+w) x_i x_j)}{(x_i+x_j+w x_i x_j)(1-x_i x_j)} \right)_{1 \le i < j \le n} & \Block{2-1}{\left( 1 \right)_{1 \le i \le n}} & \Block{2-1}{ \left( \frac{1-x_i}{x_i} \left( \prod\limits_{\myatop{1 \le p \le n}{p \neq i}} \frac{1+ w x_i + x_i x_p}{x_i + x_p + w x_i x_p} \right) \right)_{1 \le i \le n}}  & \Block{2-1}{(0)_{1 \le i \le n}} \\ \cline{1}
						& & & \\ \cline{2-4}
						& & 0 & 1 \\ \cline{3-4}
						& & & \prod_{k=1}^{n} \frac{1}{x_k}
					\end{NiceArray}
				\end{pmatrix}$}}\\
		= \pf_{1 \le i < j \le n+1} {\scalebox{\myscale}{$\begin{pmatrix}
					\begin{NiceArray}{@{}c|c@{}}[corners]
						\left( \frac{(x_j-x_i)(1+(1+w) x_i x_j)}{(x_i+x_j+w x_i x_j)(1-x_i x_j)} \right)_{1 \le i < j \le n} & \Block{2-1}{\left(1 \right)_{1 \le i \le n}} \\ \cline{1}
						&
					\end{NiceArray}
				\end{pmatrix}$}}.
	\end{multline}
	To this end, we perform similar elementary row and column operations as in the even case on the matrix underlying the Pfaffian in the left-hand side of the previous equation. We multiply the $i$th column by
	\begin{equation*}
		-\frac{1+x_i}{x_i} \prod_{\myatop{1 \le p \le n}{p \neq i}} \frac{1-x_i x_p}{x_i -x_p}
	\end{equation*}
	and add it to the penultimate column for every $i\in\{1,\dots,n\}$; proceed analogously for the rows. We then obtain for the left-hand side of \eqref{eq:2Pfaffians_Odd} the following due to Lemma~\ref{lem:MatrixOperationEven}:	
	\begin{equation}\label{eq:Pfaffian_odd}
		\pf_{1 \le i < j \le n+3}{\scalebox{\myscale}{$\begin{pmatrix}
					\begin{NiceArray}{@{}c|c|c|c@{}}[corners]
						\left( \frac{(x_j-x_i)(1+(1+w) x_i x_j)}{(x_i+x_j+w x_i x_j)(1-x_i x_j)} \right)_{1 \le i < j \le n} & \Block{2-1}{\left( 1 \right)_{1 \le i \le n}} & \Block{2-1}{\left(-1+\prod_{k=1}^{n} \frac{1}{x_k} \right)_{1 \le i \le n}}  & \Block{2-1}{(0)_{1 \le i \le n}} \\ \cline{1}
						& & & \\ \cline{2-4}
						& & 1+\prod_{k=1}^{n} \frac{1}{x_k} & 1 \\ \cline{3-4}
						& & & \prod_{k=1}^{n} \frac{1}{x_k}
					\end{NiceArray}
				\end{pmatrix}$}}.
	\end{equation}
	We continue with further row and column operations regarding \eqref{eq:Pfaffian_odd}: we add $1-\prod_{k=1}^{n}\frac{1}{x_k}$ times the $(n+1)$th column and $-(1+\prod_{k=1}^{n}\frac{1}{x_k})$ times the $(n+3)$th column to the $(n+2)$th column; similarly for the rows. Finally, we subtract the $(n+1)$ row and column from the $(n+2)$th row and column, respectively. We eventually obtain	
	\begin{equation*}
		\pf_{1 \le i < j \le n+3}{\scalebox{\myscale}{$\begin{pmatrix}
					\begin{NiceArray}{@{}c|c|c|c@{}}[corners]
						\left( \frac{(x_j-x_i)(1+(1+w) x_i x_j)}{(x_i+x_j+w x_i x_j)(1-x_i x_j)} \right)_{1 \le i < j \le n} & \Block{2-1}{\left( 1 \right)_{1 \le i \le n}} & \Block{3-1}{\left(0 \right)_{1 \le i \le n+1}}  & \Block{3-1}{(0)_{1 \le i \le n+1}} \\ \cline{1}
						& & & \\ \cline{2-2}
						& & &  \\ \cline{3-4}
						& & & 1
					\end{NiceArray}
				\end{pmatrix}$}}.
	\end{equation*}
	Expanded along the last column, this expression finally coincides with the right-hand side of \eqref{eq:2Pfaffians_Odd}. 	
\end{proof}

As previously mentioned, we now provide a sketch of the proof using the Izergin--Korepin technique.

\begin{proof}[Sketch of the second proof of Theorem~\ref{thm:Robbins_Littlewood}]
	Let $F(x_1,\dots,x_n)$ denote a multivariate function in $x_1,\dots,x_n$. Then consider the following four properties:
	\begin{enumerate}[label=(\roman*)]
		\item\label{enum:IK1} The expression $F(x_1,\dots,x_n)$ is symmetric in $x_1,\dots,x_n$.
		\item\label{enum:IK2}  It holds that $F(x_1)=\frac{1}{1-x_1}$.
		\item\label{enum:IK3} The normalized function
		\begin{equation*}
			\widetilde{F}(x_1,\dots,x_n) \coloneqq \prod_{i=1}^{n} (1-x_i) \prod_{1 \le i < j \le n} (1-x_i x_j)F(x_1,\dots,x_n)
		\end{equation*} is a rational function in $x_n$ whose numerator is of degree 	of at most~$n$ and whose denominator is of degree at most~$1$.
		\item\label{enum:IK4} We have the following $n+1$ evaluations of $\widetilde{F}(x_1,\dots,x_n)$ for $n > 1$:
		\begin{itemize}
			\item $\left.\widetilde{F}(x_1,\dots,x_n)\right\vert_{x_n=0} = \widetilde{F}(x_1,\dots,x_{n-1})$;
			\item $\left.\widetilde{F}(x_1,\dots,x_n)\right\vert_{x_n=1} = \prod_{k=1}^{n-1} (1 + x_k + w x_k) \widetilde{F}(x_1,\dots,x_{n-1})$;
			\item for all $1 \le i \le n-1$, 
			\begin{multline*}
				\qquad\qquad\left.\widetilde{F}(x_1,\dots,x_n)\right\vert_{x_n=x_i^{-1}} \\
				= \frac{w+2}{x_i^{n-2}} \prod_{\myatop{1 \le k \le n-1}{k \neq i}} (1 + w x_k + x_k x_i)(x_k + x_i + w x_k x_i) \widetilde{F}(x_1,\dots,\widehat{x_i},\dots,x_{n-1}).
			\end{multline*}
		\end{itemize}	
	\end{enumerate}
	
	These properties uniquely determine the function $F(x_1,\dots,x_n)$. It is straightforward to check that the right-hand side of \eqref{eq:Robbins_Littlewood} fulfils these properties, which is left to the reader.
	
	Hence, it remains to show that the left-hand side of \eqref{eq:Robbins_Littlewood} fulfils the same properties. For this purpose, let $F(x_1,\dots,x_n)$ denote the left-hand side of \eqref{eq:Robbins_Littlewood} for the rest of the proof.
	
	Ad~\ref{enum:IK1}: We readily obtain \ref{enum:IK1} by recalling that $F(x_1,\dots,x_n)$ equals \eqref{eq:asym_expression} and that every antisymmetric function in $x_1,\dots,x_n$ can be written as the product of $\prod_{1\le i<j\le n}(x_{j}-x_{i})$ and a symmetric function in $x_1,\dots,x_n$.
	
	\medskip
	
	Ad~\ref{enum:IK2}: The initial condition immediately follows from the definition of $F$.
	
	\medskip
	
	Ad~\ref{enum:IK3}: As a consequence of \eqref{eq:asym_expression}, we can write $\widetilde{F}(x_1,\dots,x_n)$ as
	\begin{multline}\label{eq:rec_Littlewood_tilde}
		\left(1-\prod_{i=1}^n x_i\right)^{-1}\\
		\times \sum\limits_{k=1}^n (1-x_k) \prod\limits_{\myatop{1 \le i \le n}{i \neq k}} \frac{(x_i x_k+1+w x_k) (1-x_i x_k) x_i}{x_i-x_k}
		\widetilde{F}(x_1,\ldots,\widehat{x_k},\ldots,x_n).
	\end{multline}
	We show that $(1-\prod_{i=1}^n x_i) \prod_{1 \le i < j \le n} (x_j-x_i) \widetilde{F}(x_1,\dots,x_n)$ is an antisymmetric polynomial in $x_1,\dots,x_n$ and, thus, divisible by $\prod_{1 \le i < j \le n} (x_j - x_i)$.
	To this end, let $1\le p < q \le n$. We see that $(1-\prod_{i=1}^n x_i) \prod_{1 \le i < j \le n} (x_j-x_i) \widetilde{F}(x_1,\dots,x_n)$ is equal to
	\begin{equation*}
		\begin{aligned}[t]
			\sum\limits_{\myatop{1 \le k \le n}{k \neq p,q}} & (-1)^{k-1} (1-x_k)\\
			& \times \prod_{\myatop{1 \le i < j \le n}{i,j \neq k}} (x_j-x_i) \prod\limits_{\myatop{1 \le i \le n}{i \neq k}} (x_i x_k+1+w x_k) (1-x_i x_k) x_i
			\widetilde{F}(x_1,\ldots,\widehat{x_k},\ldots,x_n)\\
			+ & (-1)^{p-1} (1-x_p)\\
			& \times \prod_{\myatop{1 \le i < j \le n}{i,j \neq p}} (x_j-x_i) \prod\limits_{\myatop{1 \le i \le n}{i \neq p}} (x_i x_p+1+w x_p) (1-x_i x_p) x_i
			\widetilde{F}(x_1,\ldots,\widehat{x_p},\ldots,x_n)\\
			+ & (-1)^{q-1} (1-x_q)\\
			& \times \prod_{\myatop{1 \le i < j \le n}{i,j \neq q}} (x_j-x_i) \prod\limits_{\myatop{1 \le i \le n}{i \neq q}} (x_i x_q+1+w x_q) (1-x_i x_q) x_i
			\widetilde{F}(x_1,\ldots,\widehat{x_q},\ldots,x_n).
		\end{aligned}
	\end{equation*}
	Interchanging $x_p$ and $x_q$ in the previous expression reverses the sign, that is, we eventually obtain 
	$-(1-\prod_{i=1}^n x_i) \prod_{1 \le i < j \le n} (x_j-x_i) \widetilde{F}(x_1,\dots,x_n)$. This holds because of the symmetry of $\widetilde{F}$ and since
	\begin{equation*}
		\left.\prod_{1 \le i < j \le n} (x_j-x_i) \right\vert_{x_q \leftrightarrow x_p}= - \prod_{1 \le i < j \le n} (x_j-x_i)
	\end{equation*}
	as well as
	\begin{equation*}
		\prod_{\myatop{1 \le i < j \le n}{i,j \neq p}} (x_j-x_i) \left.\vphantom{\prod_{1 \le i < j \le n} (x_j-x_i)}\right\vert_{x_q \leftrightarrow x_p}= (-1)^{q-p-1} \prod_{\myatop{1 \le i < j \le n}{i,j \neq q}} (x_j-x_i).
	\end{equation*}
	It follows that
	\begin{equation*}
		\left(1-\prod_{i=1}^n x_i \right) \prod_{i=1}^{n} (1-x_i) \prod_{1 \le i < j \le n} (1-x_i x_j)F(x_1,\dots,x_n)
	\end{equation*}
	is a polynomial in $x_1,\dots,x_n$. In particular, it is a polynomial in $x_n$ of degree at most~$n+1$ as a consequence of \eqref{eq:rec_Littlewood_tilde}.
	
	\medskip
	
	Ad~\ref{enum:IK4}: The evaluation for $x_n=0$ directly follows from \eqref{eq:rec_Littlewood_tilde}. The case $x_n=1$ is dealt with by induction on the number of variables as follows: The base case immediately follows from the definition of $\widetilde{F}$ and \ref{enum:IK2}. Next, setting $x_n=1$ in \eqref{eq:rec_Littlewood_tilde} yields
	\begin{multline*}
		\left(1-\prod_{i=1}^{n-1} x_i \right)^{-1} \sum\limits_{k=1}^{n-1} (1-x_k) (x_k+1+w x_k) \\
		\times \prod\limits_{\myatop{1 \le i \le n-1}{i \neq k}} \frac{(x_i x_k+1+w x_k) (1-x_i x_k) x_i}{x_i-x_k}
		\left(\left.\widetilde{F}(x_1,\ldots,\widehat{x_k},\ldots,x_n)\right\vert_{x_n=1}\right).
	\end{multline*}
	By assuming the induction hypothesis, the previous expression is equal to
	\begin{multline*}
		\left(1-\prod_{i=1}^{n-1} x_i \right)^{-1} \sum\limits_{k=1}^{n-1} (1-x_k) (x_k+1+w x_k) \\
		\times \prod\limits_{\myatop{1 \le i \le n-1}{i \neq k}} \frac{(x_i x_k+1+w x_k) (1-x_i x_k) x_i}{x_i-x_k}  \prod_{\myatop{1 \le l \le n-1}{l \neq k}} (1+x_l+w x_l) \widetilde{F}(x_1,\ldots,\widehat{x_k},\ldots,x_{n-1}),
	\end{multline*}
	which further simplifies to
	\begin{multline*}
		\prod_{1 \le l \le n-1} (1+x_l+w x_l) \left(1-\prod_{i=1}^{n-1} x_i \right)^{-1} \sum\limits_{k=1}^{n-1} (1-x_k) \\
		\times \prod\limits_{\myatop{1 \le i \le n-1}{i \neq k}} \frac{(x_i x_k+1+w x_k) (1-x_i x_k) x_i}{x_i-x_k}
		\widetilde{F}(x_1,\ldots,\widehat{x_k},\ldots,x_{n-1})\\
		=\prod_{1 \le l \le n-1} (1+x_l+w x_l) \widetilde{F}(x_1,\dots,x_{n-1})
	\end{multline*}
	as desired.
	
	For the remaining $n-1$ evaluations, we use again induction on the number of variables. The base case is readily checked as before. Furthermore, it suffices to consider $x_n=x_{n-1}^{-1}$ because of the symmetry of $\widetilde{F}$. Given that case, we obtain
	\begin{multline*}
		\left(1-\prod_{i=1}^{n-2} x_i \right)^{-1} \sum\limits_{k=1}^{n-2} (1-x_k) \frac{(1+w x_k+x_k x_{n-1})(x_k+x_{n-1}+w x_k x_{n-1})}{x_{n-1}}\\
		\times\prod\limits_{\myatop{1 \le i \le n-2}{i \neq k}} \frac{(x_i x_k+1+w x_k) (1-x_i x_k) x_i}{x_i-x_k}
		\left(\left.\widetilde{F}(x_1,\ldots,\widehat{x_k},\ldots,x_n)\right\vert_{x_n=x_{n-1}^{-1}}\right),
	\end{multline*}
	which follows from
	\begin{multline*}
		\frac{(x_{n-1} x_k +1+w x_k)(1-x_{n-1} x_k) x_{n-1}}{x_{n-1}-x_k} \cdot\frac{(x_{n-1}^{-1} x_k +1+w x_k)(1-x_{n-1}^{-1} x_k) x_{n-1}^{-1}}{x_{n-1}^{-1}-x_k}\\
		= \frac{(x_{n-1} x_k +1+w x_k)(x_{n-1}+x_k+w x_{n-1} x_k)}{x_{n-1}}.
	\end{multline*}
	By assuming the induction hypothesis, we proceed with
	\begin{multline*}
		\left(1-\prod_{i=1}^{n-2} x_i \right)^{-1} \sum\limits_{k=1}^{n-2} (1-x_k) \frac{(1+w x_k+x_k x_{n-1})(x_k+x_{n-1}+w x_k x_{n-1})}{x_{n-1}}\\ \times\prod\limits_{\myatop{1 \le i \le n-2}{i \neq k}} \frac{(x_i x_k+1+w x_k) (1-x_i x_k) x_i}{x_i-x_k}\\
		\times \frac{w+2}{x_{n-1}^{n-3}} \prod_{\myatop{1 \le l \le n-2}{l \neq k}} (1 + w x_l + x_l x_{n-1})(x_l + x_{n-1} + w x_l x_{n-1}) \widetilde{F}(x_1,\dots,\widehat{x_k},\dots,x_{n-2}),
	\end{multline*}
	which finally equals
	\begin{multline*}
		\frac{w+2}{x_{n-2}^{n-1}} \prod_{1 \le l \le n-2} (1 + w x_l + x_l x_{n-1})(x_l + x_{n-1} + w x_l x_{n-1})\\
		\times \left(1-\prod_{i=1}^{n-2} x_i \right)^{-1} \sum\limits_{k=1}^{n-2} (1-x_k) \prod\limits_{\myatop{1 \le i \le n-2}{i \neq k}} \frac{(x_i x_k+1+w x_k) (1-x_i x_k) x_i}{x_i-x_k}\\ \times \widetilde{F}(x_1,\dots,\widehat{x_k},\dots,x_{n-2})\\
		= \frac{w+2}{x_{n-1}^{n-2}} \prod_{1 \le l \le n-2} (1 + w x_l + x_l x_{n-1})(x_l + x_{n-1} + w x_l x_{n-1}) \widetilde{F}(x_1,\dots,x_{n-2}).
	\end{multline*}
\end{proof}

As a consequence of Theorem~\ref{thm:Robbins_Littlewood}, Corollary~\ref{cor:Littlewood} follows by straightforward algebraic manipulation, as shown next.

\begin{proof}[Proof of Corollary~\ref{cor:Littlewood}] By definition, 
	\begin{multline*}
		\sum_{k_1 \ge \dots \ge k_n \ge 0} R^{\ast}_{(k_{1},\ldots,k_{n})}(x_{1},\ldots,x_{n};1,1,w) \\ 
		= \frac{
			\asym_{x_1,\ldots,x_n} \left[ \prod_{1 \le i < j \le n} (x_i x_j +1 + w x_i) \sum_{k_1 \ge \dots \ge k_n \ge 0 } \prod_{i=1}^n x_i^{k_i} \right]}
		{\prod_{1 \le i < j \le n} (x_j - x_i)}.
	\end{multline*} 
	Using the geometric series expansion, this is further equal to
	\begin{multline*} 
		\frac{
			\asym_{x_1,\ldots,x_n} \left[ \prod_{1 \le i < j \le n} (x_i x_j +1 + w x_i) 
			\prod_{i=1}^{n} \left( \frac{1}{1 - \prod_{j=1}^i x_j} \right) \right]}
		{\prod_{1 \le i < j \le n} (x_j - x_i)} \\
		= (-1)^n \frac{\asym_{x_1,\ldots,x_n} \left[ \prod_{1 \le i < j \le n} (x_i^{-1} x_j^{-1} +1 + w x_j^{-1}) 
			\prod_{i=1}^{n} \left( \frac{\prod_{j=1}^i x_j^{-1}}{1 - \prod_{j=1}^i x_j^{-1}} \right) \right]}
		{\prod_{1 \le i < j \le n} (x_i^{-1} - x_j^{-1})}.
	\end{multline*} 
	The transformation $x_i \mapsto x_{n+1-i}$ for all $i$ yields
	\begin{equation*}
		(-1)^n \frac{\asym_{x_1,\ldots,x_n} \left[ \prod_{1 \le i < j \le n} (x_i^{-1} x_j^{-1} +1 + w x_i^{-1}) 
			\prod_{i=1}^{n} \left( \frac{ \prod_{j=i}^n x_j^{-1}}{1 - \prod_{j=i}^n x_j^{-1}} \right) \right]}
		{\prod_{1 \le i < j \le n} (x_j^{-1} - x_i^{-1})},
	\end{equation*}
	which equals
	\begin{equation*}
		(-1)^n \prod_{i=1}^n x_i^{-1} \sum_{0 \le k_1 < \dots < k_n} R^{\ast}_{(k_{1},\ldots,k_{n})}(x_{1}^{-1},\ldots,x_{n}^{-1};1,1,w).
	\end{equation*}
	Now Theorem~\ref{thm:Robbins_Littlewood} implies the result. 
\end{proof}

\section{Six-vertex model configurations}
\label{sec:six-vertex}

In this section, we shall see that the right-hand side of \eqref{eq:Robbins_Littlewood} in Theorem~\ref{thm:Robbins_Littlewood} can be interpreted in terms of six-vertex model configurations of DSASMs.

We can recover a DSASM from its upper triangular part due to the diagonal symmetry. Therefore, we shall define six-vertex model configurations on triangular grids that are in bijective correspondence with DSASMs. Let $\mathcal{T}_n$ be the following triangular grid with $n$ top vertices of degree~$1$, $n$ right boundary vertices of degree~$1$, $n$ left boundary vertices of degree~$2$ and $\binom{n}{2}$ bulk vertices of degree~$4$:

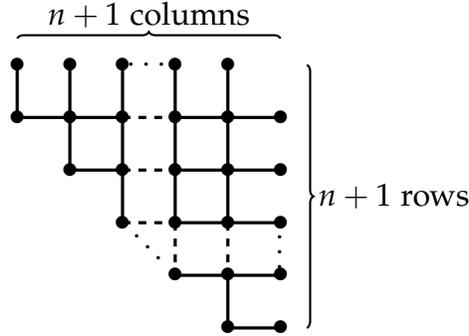
\begin{center}
	\begin{tikzpicture}[scale=.35]
		\node at (0,2) {$\bullet$};
		\node at (2,2) {$\bullet$};
		\node at (4,2) {$\bullet$};
		\node at (6,2) {$\bullet$};
		\node at (8,2) {$\bullet$};
		
		\node at (0,0) {$\bullet$};
		\node at (2,0) {$\bullet$};
		\node at (4,0) {$\bullet$};
		\node at (6,0) {$\bullet$};
		\node at (8,0) {$\bullet$};
		
		\node at (2,-2) {$\bullet$};
		\node at (4,-2) {$\bullet$};
		\node at (6,-2) {$\bullet$};
		\node at (8,-2) {$\bullet$};
		
		\node at (4,-4) {$\bullet$};
		\node at (6,-4) {$\bullet$};
		\node at (8,-4) {$\bullet$};
		
		\node at (6,-6) {$\bullet$};
		\node at (8,-6) {$\bullet$};
		
		\node at (8,-8) {$\bullet$};
		
		\node at (10,0) {$\bullet$};
		\node at (10,-2) {$\bullet$};
		\node at (10,-4) {$\bullet$};
		\node at (10,-6) {$\bullet$};
		\node at (10,-8) {$\bullet$};
		
		\node at (5,2) {$\cdots$};
		\node at (5,-5) {\rotatebox{-45}{$\cdots$}};
		\node at (10,-5) {\rotatebox{90}{$\cdots$}};
		
		\draw[very thick] (0,2) -- (0,0) -- (4,0);
		\draw[very thick,dashed] (4,0) -- (6,0);
		\draw[very thick] (6,0) -- (10,0);
		
		\draw[very thick] (2,2) -- (2,-2) -- (4,-2);
		\draw[very thick,dashed] (4,-2) -- (6,-2);
		\draw[very thick] (6,-2) -- (10,-2);
		
		\draw[very thick] (4,2) -- (4,-4);
		\draw[very thick,dashed] (4,-4) -- (6,-4);
		\draw[very thick] (6,-4) -- (10,-4);
		
		\draw[very thick] (6,2) -- (6,-4);
		\draw[very thick,dashed] (6,-4) -- (6,-6);
		\draw[very thick] (6,-6) -- (10,-6);
		
		\draw[very thick] (8,2) -- (8,-4);
		\draw[very thick,dashed] (8,-4) -- (8,-6);
		\draw[very thick] (8,-6) -- (8,-8) -- (10,-8);
		
		\path[draw,decorate,decoration=brace,thick] (11,2) -- (11,-8)
		node[midway,right]{$n+1$ rows};
		
		\path[draw,decorate,decoration=brace,thick] (0,3) -- (10,3)
		node[midway,above]{$n+1$ columns};
	\end{tikzpicture}
\end{center}

Recall that a directed graph is a graph where each edge is equipped with an orientation. The \deff{indegree} of a vertex is the number of edges leading towards that vertex and the $\deff{outdegree}$ is the number of edges leading away from it.

\begin{dfn}
	A \deff{six-vertex model configuration} on $\mathcal{T}_n$ 
	is an orientation of the edges of $\mathcal{T}_n$ such that
	\begin{itemize}
		\item top vertices have indegree~$1$;
		\item right boundary vertices have outdegree~$1$;
		\item bulk vertices have both indegree and outdegree equal to $2$.
	\end{itemize}
\end{dfn}

The name six-vertex model stems from the fact that there are six possibilities for different orientations of bulk vertices since a bulk vertex is incident to four edges, two of which are leading in and two are leading out, amounting to a total of $\binom{4}{2}=6$ possibilities.

We denote the set of all six-vertex model configurations on $\mathcal{T}_n$ by $\SV(n)$. Six-vertex model configurations on $\mathcal{T}_n$ are in simple bijective correspondence with DSASMs of order~$n$ \cite{BFK23}; it follows from the well-known bijection between ASMs and six-vertex model configurations
with so-called domain-wall boundary conditions when restricted to DSASMs.

\begin{exm}
	\label{exm:DSASM}
	This example shows a DSASM of order~$5$ and the corresponding six-vertex model configuration. How local configurations around left boundary vertices and bulk vertices correspond to the entries of the matrix can be seen in Table~\ref{tab:six-vertex weights}. See Table~\ref{tab:bijection} at the end of the Section~\ref{sec:highest_term} for further examples.
	\begin{center}
		$\begin{pmatrix}
			0 & 0 & 1 & 0 & 0\\
			0 & 1 & 0 & 0 & 0\\ 
			1 & 0 &-1 & 1 & 0\\ 
			0 & 0 & 1 & 0 & 0\\
			0 & 0 & 0 & 0 & 1
		\end{pmatrix}$
		$\longleftrightarrow$
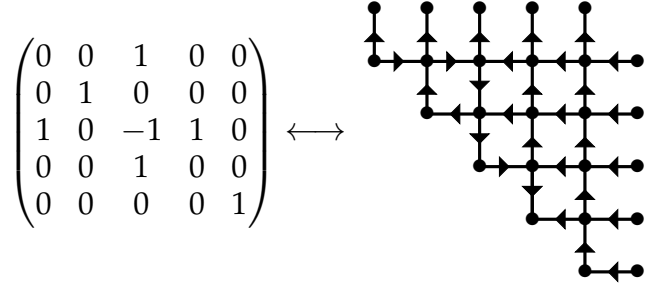
		\begin{tikzpicture}[scale=.35,baseline=(current bounding box.center),every node/.style={sloped,allow upside down}]
			\node at (0,2) {$\bullet$};
			\node at (2,2) {$\bullet$};
			\node at (4,2) {$\bullet$};
			\node at (6,2) {$\bullet$};
			\node at (8,2) {$\bullet$};
			
			\node at (0,0) {$\bullet$};
			\node at (2,0) {$\bullet$};
			\node at (4,0) {$\bullet$};
			\node at (6,0) {$\bullet$};
			\node at (8,0) {$\bullet$};
			
			\node at (2,-2) {$\bullet$};
			\node at (4,-2) {$\bullet$};
			\node at (6,-2) {$\bullet$};
			\node at (8,-2) {$\bullet$};
			
			\node at (4,-4) {$\bullet$};
			\node at (6,-4) {$\bullet$};
			\node at (8,-4) {$\bullet$};
			
			\node at (6,-6) {$\bullet$};
			\node at (8,-6) {$\bullet$};
			
			\node at (8,-8) {$\bullet$};
			
			\node at (10,0) {$\bullet$};
			\node at (10,-2) {$\bullet$};
			\node at (10,-4) {$\bullet$};
			\node at (10,-6) {$\bullet$};
			\node at (10,-8) {$\bullet$};
			
			\draw[very thick] (0,0) --node{\midarrow} (0,2);
			\draw[very thick] (2,0) --node{\midarrow} (2,2);
			\draw[very thick] (4,0) --node{\midarrow} (4,2);
			\draw[very thick] (6,0) --node{\midarrow} (6,2);
			\draw[very thick] (8,0) --node{\midarrow} (8,2);
			
			\draw[very thick] (10,0) --node{\midarrow} (8,0);
			\draw[very thick] (10,-2) --node{\midarrow} (8,-2);
			\draw[very thick] (10,-4) --node{\midarrow} (8,-4);
			\draw[very thick] (10,-6) --node{\midarrow} (8,-6);
			\draw[very thick] (10,-8) --node{\midarrow} (8,-8);
			
			\draw[very thick] (0,0) --node{\midarrow} (2,0);
			
			\draw[very thick] (2,-2) --node{\midarrow} (2,0);
			\draw[very thick] (4,-2) --node{\midarrow} (2,-2);
			
			\draw[very thick] (4,-2) --node{\midarrow} (4,-4);
			\draw[very thick] (4,-4) --node{\midarrow} (6,-4);
			
			\draw[very thick] (6,-4) --node{\midarrow} (6,-6);
			\draw[very thick] (8,-6) --node{\midarrow} (6,-6);
			
			\draw[very thick] (8,-8) --node{\midarrow} (8,-6);
			
			\draw[very thick] (2,0) --node{\midarrow} (4,0);
			\draw[very thick] (4,0) --node{\midarrow} (4,-2);
			\draw[very thick] (6,0) --node{\midarrow} (4,0);
			\draw[very thick] (8,0) --node{\midarrow} (6,0);
			
			\draw[very thick] (6,-2) --node{\midarrow} (4,-2);
			\draw[very thick] (6,-2) --node{\midarrow} (6,0);
			\draw[very thick] (8,-2) --node{\midarrow} (6,-2);
			\draw[very thick] (8,-2) --node{\midarrow} (8,0);
			
			\draw[very thick] (6,-4) --node{\midarrow} (6,-2);
			\draw[very thick] (6,-4) --node{\midarrow} (6,-6);
			\draw[very thick] (8,-4) --node{\midarrow} (8,-2);
			\draw[very thick] (8,-4) --node{\midarrow} (6,-4);
			
			\draw[very thick] (8,-6) --node{\midarrow} (8,-4);
			
		\end{tikzpicture}
	\end{center}
\end{exm}

We associate a weight to a given six-vertex model configuration by assigning a weight to each vertex; the weight of the six-vertex model configuration is then the product of the weights of all vertices. We first assign matrix coordinates to the vertices in $\mathcal{T}_n$: a vertex in the $(i+1)$th row from top and in the $j$th column from left is said to be located at position $(i,j)$. The weights of the vertices can then be determined as follows with the help of Table~\ref{tab:six-vertex weights}, where $\tau(x) \coloneq \sqrt{-1-w x-x^2}$:
\begin{itemize}
	\item Top vertices and right boundary vertices have weight $1$.
	\item Bulk vertices at position $(i,j)$ with local configuration~$c$ have weight $W(c;x_i,x_j)$.
	\item Left boundary vertices at $(i,i)$ with local configuration~$c$ have weight $W(c;x_i)$.
\end{itemize}
The sum of these weights over all six-vertex model configurations in $\SV(n)$ is called the \deff{partition function of DSASMs of order~$n$}, denoted by $Z_{\DSASM} (x_1,\dots,x_n)$.

\begin{table}[htbp]
	\centering
	\begin{tabular}{llc}\toprule
		bulk weights & left boundary weights & matrix entries\\\midrule\medskip
		$W\Big(\raisebox{-1.5ex}{\wone};x,y\Big)=\tau(x) \tau(y)$ & $W\Big(\raisebox{-.25ex}{\wlone};x\Big)=1$ & $1$\\\medskip
		$W\Big(\raisebox{-1.5ex}{\wmone};x,y\Big)=\tau(x) \tau(y)$ & $W\Big(\raisebox{-.25ex}{\wlmone};x\Big)=- 1$ & $-1$\\\medskip
		$W\Big(\raisebox{-1.5ex}{\wne};x,y\Big) = x + y + w x y$ & $W\Big(\raisebox{-.25ex}{\wlout};x\Big)=\sqrt{w+2} \frac{x}{\tau(x)}$ & $0$\\\medskip
		$W\Big(\raisebox{-1.5ex}{\wsw};x,y\Big) = x + y + w x y$ & $W\Big(\raisebox{-.25ex}{\wlin};x\Big)=\sqrt{w+2} \frac{x}{\tau(x)}$ & $0$\\\medskip
		$W\Big(\raisebox{-1.5ex}{\wnw};x,y\Big) = 1-x y$ & & $0$\\\medskip
		$W\Big(\raisebox{-1.5ex}{\wse};x,y\Big) = 1-x y$ & & $0$\\\bottomrule
	\end{tabular}
	\caption[Vertex weights]{Vertex weights in the six-vertex model.\protect\footnotemark}
	\label{tab:six-vertex weights}
\end{table}

It can be seen that the weight of each six-vertex model configuration on $\mathcal{T}_n$ is indeed a polynomial in $w$ and $x_1$,\dots,$x_n$.

\begin{prop}
	\label{prop:polynomial} 
	The weight of a six-vertex model configuration on $\mathcal{T}_n$ with respect to Table~\ref{tab:six-vertex weights} is polynomial in $x_1,\dots,x_n$ and $w$.
\end{prop}

\begin{proof}
	First, we point out that all local configurations in Table~\ref{tab:six-vertex weights} except $\raisebox{-.5ex}{\wlout}$ and $\raisebox{-.5ex}{\wlin}$ are balanced in the sense that for each vertex the indegree coincides with the outdegree. Since for every directed graph the sum of all indegrees equals the sum of all outdegrees, there are as many $\raisebox{-.5ex}{\wlout}$ as $\raisebox{-.5ex}{\wlin}$ in any six-vertex model configuration, considering also the equal numbers of top and right boundary vertices in $\mathcal{T}_n$. Hence, the factor $\sqrt{w+2}$ appears with an even exponent in the weight. Consequently, the weight is a polynomial in $w$.
	
	\footnotetext{Note that the bulk weights given in Table~\ref{tab:six-vertex weights} are, up to a common factor that depend on the spectral parameters $x$ and $y$, a reparametrization of the Boltzmann weights given by the standard $R$-matrix associated to the quantum group $U_q(\widehat{\mathfrak{sl}}_2)$ of the affine Lie algebra~$\widehat{\mathfrak{sl}}_2$. Under this parametrization, $w$ plays the role of the quantum group parameter. This fact will become apparent in the proof of Theorem~\ref{thm:partition function}.}
	
	Next, we dissect $\mathcal{T}_n$ into $n$ $\raisebox{-.5ex}{\wL}$-shaped paths from $(i,n+1)$ to $(0,i)$ for all $1\le i \le n$. The corresponding entries of the DSASM along these paths sum to $1$, which follows from the diagonal symmetry and the fact that the sums of entries along rows and columns equal $1$ in any ASM. Hence, there is an odd number of local configurations corresponding to an entry of $1$ or $-1$ along each path. If the left boundary vertex of a given path is $\raisebox{-.5ex}{\wlone}$ or $\raisebox{-.5ex}{\wlmone}$, then there is an even number of bulk vertices $\raisebox{-1.5ex}{\wone}$ and $\raisebox{-1.5ex}{\wmone}$ in the same column above and in the same row to the right. However, if the left boundary vertex is $\raisebox{-.5ex}{\wlout}$, there is an even number of bulk vertices$\raisebox{-1.5ex}{\wone}$ and $\raisebox{-1.5ex}{\wmone}$ in the column above and an odd number in the row to the right. Conversely, if the left boundary vertex is $\raisebox{-.5ex}{\wlin}$, there is an odd number of these bulk vertices in the column above and an even number in the row to the right. The factor $\tau(x)^{-1}$, contributed by each of these left boundary vertices $\raisebox{-.5ex}{\wlout}$ and $\raisebox{-.5ex}{\wlin}$, is cancelled by one of the vertices in the part of the path with the odd number of bulk vertices corresponding to $1$ or $-1$. It is straightforward to see that, in total, every factor $\tau(x)$ appears with a positive even integer, resulting in a polynomial weight in $x_1,\dots,x_n$.
\end{proof}

The following theorem establishes a closed expression for the partition function of DSASMs in our setting.

\begin{thm}
	\label{thm:partition function}
	The partition function~$Z_{\DSASM} (x_1,\dots,x_n)$ of DSASMs of order~$n$ with respect to the weights given in Table~\ref{tab:six-vertex weights} is
	\begin{multline}\label{eq:partition function}
		\prod_{1 \le i < j \le n} \frac{(1-x_i x_j) (x_i+x_j+w x_i x_j)}{x_j -x_i}\\
		\times \pf_{\chi_{\mathrm{even}}(n) \le i < j \le n} \left(
		\begin{cases}
			1, & i=0,\\
			\frac{(x_j-x_i)(1+(1+w) x_i x_j)}{(x_i+x_j+w x_i x_j)(1-x_i x_j)},& i\ge 1.
		\end{cases} \right)
	\end{multline}
\end{thm}

\begin{proof}
	We consider the substitution $w=q^4+q^{-4}$ as well as $x_i=-q^2 (1+q^2 u_i^2) (q^6 + u_i^2)^{-1}$ for all $1 \le i \le n$, under which
	$\tau(x_i)$ maps to $\rho(u_i)\coloneqq\frac{(1-q^8) u_i}{q (q^6 +u_i^2)}$. The vertex weights in Table~\ref{tab:six-vertex weights} thus become
	\begin{align*}
		W\Big(\raisebox{-1.5ex}{\wone};u_i,u_j\Big)&=W\Big(\raisebox{-1.5ex}{\wmone};u_i,u_j\Big)=\rho(u_i) \rho(u_j),\\
		W\Big(\raisebox{-1.5ex}{\wne};u_i,u_j\Big)&=W\Big(\raisebox{-1.5ex}{\wsw};u_i,u_j\Big)=\frac{q^2 (1-q^4 u_i^2 u_j^2)}{u_i u_j (1-q^8)} \rho(u_i) \rho(u_j),\\
		W\Big(\raisebox{-1.5ex}{\wnw};u_i,u_j\Big)&=W\Big(\raisebox{-1.5ex}{\wse};u_i,u_j\Big)=\frac{q^2 (u_i^2 u_j^2 - q^4)}{u_i u_j  (1-q^8)} \rho(u_i) \rho(u_j),\\
		W\Big(\raisebox{-.25ex}{\wlout};u_i\Big)&=W\Big(\raisebox{-.25ex}{\wlin};u_i\Big)= -\frac{q (1+q^2 u_i^2)}{u_i (1-q^4)};
	\end{align*}
	all other vertex weights remain unchanged.
	
	Using the following notation, namely 
	\begin{equation*}
		\overline{x} \coloneqq x^{-1}, \quad \sigma(x) \coloneqq x - \overline{x} \quad \text{and} \quad \widehat{\sigma}(x) \coloneqq \sigma(x)/\sigma(q^4),
	\end{equation*}
	and setting $\alpha = - \overline{q} (1-q^4)$, $\beta = -\alpha$, $\gamma = \delta = q $ as well as $\phi(u) = \alpha^{-1} \sigma(q u)^{-1}$, we see that
	\begin{align*}
		\frac{q^2 (1-q^4 u_i^2 u_j^2)}{u_i u_j (1-q^8)} &= \widehat{\sigma}(q^2 u_i u_j),\\ 
		\frac{q^2 (u_i^2 u_j^2 - q^4)}{u_i u_j (1-q^8)} &= \widehat{\sigma}(q^2 \overline{u_i u_j}) \quad\text{and}\\
		-\frac{q (1+q^2 u_i^2)}{u_i (1-q^4)} &= \gamma \sigma(q^2 u_i^2) \phi(u_i).
	\end{align*}
	
	It follows from \cite[Theorem~11]{BFK23} that the partition function with respect to these weights is given by
	\begin{multline*}
		\prod_{1 \le i < j \le n} \rho(u_i) \rho(u_j) \frac{(1-q^4 u_i^2 u_j^2)(q^4 - u_i^2 u_j^2)}{u_i u_j(1-q^8)(u_i^2 - u_j^2)}\\ \times\pf_{\chi_{\mathrm{even}}(n) \le i < j \le n} \left(
		\begin{cases}
			1, & i=0,\\
			\frac{q^2 (u_i^2-u_j^2)(1+q^8+q^2(1+q^4)(u_i^2+u_j^2)+(1+q^8)u_i^2 u_j^2)}{(1-q^4)(q^4-u_i^2 u_j^2)(1-q^4 u_i^2 u_j^2)},& i\ge 1,
		\end{cases} \right)
	\end{multline*}
	which coincides with \eqref{eq:partition function} under our substitution.
\end{proof}

Theorem~\ref{thm:comb_interpretation_RHS} is now an immediate consequence of Theorem~\ref{thm:Robbins_Littlewood} and Theorem~\ref{thm:partition function}.

Note that the factors $\rho(u_i) \rho(u_j)$ of the bulk weights are not necessary for computing the partition function~$Z_{\DSASM} (x_1,\dots,x_n)$. However, by incorporating these additional factors, we are able to reveal the striking analogy between the identity given in Theorem~\ref{thm:comb_interpretation_RHS} and the classical Littlewood identity~\eqref{eq:Littlewood}.

\section{\texorpdfstring{The coefficient of the highest term in the polynomial expansion of $Z_{\DSASM} (x_1,\dots,x_n)$}{The coefficient of the highest term in the polynomial expansion of the partition function}}
\label{sec:highest_term}

We have argued in Section~\ref{sec:six-vertex} that the partition function $Z_{\DSASM} (x_1,\dots,x_n)$ is a symmetric polynomial in $x_1,\dots,x_n$. Since $Z_{\DSASM} (x_1,\dots,x_n)$ is given by
\begin{equation*}
	\prod_{1 \le i < j \le n} \frac{(1-x_i x_j)(x_i+x_j+w x_i x_j)}{x_j-x_i}
	\pf_{\chi_{\mathrm{even}}(n) \le i < j \le n} \left(
	\begin{cases}
		1, & i=0,\\
		\frac{(x_j-x_i)(1+(1+w) x_i x_j)}{(x_i+x_j+w x_i x_j)(1-x_i x_j)},& i\ge 1,
	\end{cases} \right)
\end{equation*}
the total degree of $Z_{\DSASM} (x_1,\dots,x_n)$ is $n(n-1)$. In this section, we compute the coefficient of the leading term $x_1^{n-1} \cdots x_n^{n-1}$.

To this end, we need the following two lemmas on the manipulation of Pfaffians.

\begin{lem}\label{thm:pfaffian_transform}
	For an even positive integer~$n$ and power series~$f(x,y)$, $h(x)$ and $k(x)$ such that $f(x,y)$ is antisymmetric in $x$ and $y$, it holds that
	\begin{equation*}
		\pf_{0 \le i < j \le n-1} \left( \langle u^i v^j \rangle k(u) k(v) f(h(u)u, h(v)v) \right) = k(0)^n h(0)^{\binom{n}{2}} \pf_{0 \le i < j \le n-1} \left( \langle u^i v^j \rangle f(u,v) \right).
	\end{equation*}
\end{lem}

\begin{proof}
	This lemma directly follows from \cite[Theorem~14]{BFK23} if we set $m=2n$.
\end{proof}

\begin{lem}\label{thm:pfaffian_transform_oddcase}
	For an odd positive integer~$n$ and an antisymmetric power series~$f(x,y)$ in $x$ and $y$, it holds that
	\begin{multline}\label{eq:pfaffian_transform_oddcase}
		\pf_{1 \le i < j \le n-1} \left(\langle u^i v^j\rangle (1 \pm u)^{\delta_{i,1}} f(u,v) \right) \\
		= \pf_{0 \le i < j \le n-2} \left(\langle u^i v^j\rangle \frac{f(u,v) - (1 \mp u) f(0,v) - (1 \mp v) f(u,0)}{u v} \right).
	\end{multline}
\end{lem}

\begin{proof}
	Let $f(u,v) = \sum_{i,j \ge 0} f_{i,j} u^i v^j$ be an antisymmetric power series in $u$ and $v$. It follows that
	\begin{equation*}
		\langle u^i v^j \rangle (1 \pm u)^{\delta_{i,1}} f(u,v) =
		\begin{cases}
			f_{1,j} \pm f_{0,j}, & 1=i<j,\\
			f_{i,j}, & 1 \neq i<j.
		\end{cases}
	\end{equation*}
	Next, we consider
	\begin{equation*}
		f(u,v) \pm u f(0,v) \pm v f(u,0),
	\end{equation*}
	which is antisymmetric in $x$ and $y$ as well. We see that it has the same distribution of coefficients as $(1 \pm u)^{\delta_{i,1}} f(u,v)$. Thus, we know that the left-hand side of \eqref{eq:pfaffian_transform_oddcase} equals
	\begin{equation}\label{eq:pfaffian_transform_oddcase_a}
		\pf_{1 \le i < j \le n-1} \left(\langle u^i v^j\rangle \left( f(u,v) \pm u f(0,v) \pm v f(u,0) \right) \right).
	\end{equation}
	
	We want to shift the indices from $1 \le i < j \le n-1$ to $0 \le i < j \le n-2$. Let $g(u,v)$ be a power series in $u$ and $v$. Then
	\begin{equation*}
		\frac{g(u,v)-g(u,0)-g(0,v)+g(0,0)}{u v}
	\end{equation*}
	is also a power series satisfying
	\begin{equation*}
		\langle u^{i-1} v^{j-1} \rangle \frac{g(u,v)-g(u,0)-g(0,v)+g(0,0)}{u v} = \langle u^i v^j \rangle g(u,v).
	\end{equation*}
	
	If we set $g(u,v)$ to be $f(u,v) \pm u f(0,v) \pm v f(u,0)$, we finally conclude that \eqref{eq:pfaffian_transform_oddcase_a} is equal to the right-hand side of \eqref{eq:pfaffian_transform_oddcase}. \qedhere	
\end{proof}

We are now ready to compute the coefficient of the highest term in the polynomial expansion of $Z_{\DSASM} (x_1,\dots,x_n)$.

\begin{proof}[Proof of Theorem~\ref{thm:highest_term}]
	In order to compute the highest term in $Z_{\DSASM} (x_1,\dots,x_n)$, we have to consider the vertex weights given in Table~\ref{tab:six-vertex weights highest term}, where $\imag$ denotes the imaginary unit.
	
	\begin{table}[htbp]
		\centering
		\begin{tabular}{llc}\toprule
			bulk weights & left boundary weights & matrix entries\\\midrule\medskip
			$W\Big(\raisebox{-1.5ex}{\wone};x,y\Big)= 1$ & $W\Big(\raisebox{-.25ex}{\wlone};x\Big)=1$ & $1$\\\medskip
			$W\Big(\raisebox{-1.5ex}{\wmone};x,y\Big)= 1$ & $W\Big(\raisebox{-.25ex}{\wlmone};x\Big)=- 1$ & $-1$\\\medskip
			$W\Big(\raisebox{-1.5ex}{\wne};x,y\Big) = - w$ & $W\Big(\raisebox{-.25ex}{\wlout};x\Big)=- \imag \sqrt{w+2} $ & $0$\\\medskip
			$W\Big(\raisebox{-1.5ex}{\wsw};x,y\Big) = - w$ & $W\Big(\raisebox{-.25ex}{\wlin};x\Big)= - \imag \sqrt{w+2}$ & $0$\\\medskip
			$W\Big(\raisebox{-1.5ex}{\wnw};x,y\Big) = 1$ & & $0$\\\medskip
			$W\Big(\raisebox{-1.5ex}{\wse};x,y\Big) = 1$ & & $0$\\\bottomrule
		\end{tabular}
		\caption{Vertex weights for the highest term in the polynomial expansion of $Z_{\DSASM} (x_1,\dots,x_n)$.}
		\label{tab:six-vertex weights highest term}
	\end{table}
	
	To see this, we extract the terms with the largest contribution to the coefficient of $x_1^{n-1} \cdots x_n^{n-1}$ from the weights in Table~\ref{tab:six-vertex weights} and additionally multiply every bulk weight with $-1$. Thus, we obtain the coefficient $x_1^{n-1} \cdots x_n^{n-1}$ in $Z_{\DSASM} (x_1,\dots,x_n)$ by the following expression:
	\begin{equation}\label{eq:highest_weight}
		(-1)^{\binom{n}{2}} \sum_{\SV(n)} (-w)^{\# \raisebox{-1ex}{\scalebox{.75}{\wne}}+\# \raisebox{-1ex}{\scalebox{.75}{\wsw}}}  (-1)^{\# \scalebox{.75}{\wlmone}} (- \imag \sqrt{w+2})^{\# \scalebox{.75}{\wlout}+\# \scalebox{.75}{\wlin}}.
	\end{equation}
	The previous generating function can be computed by the following expression when setting $p=w^2$, $r=1$, $s_+=\frac{1}{-\imag\sqrt{w+2}}$, and $s_-=\frac{1}{\imag \sqrt{w+2}}$ in \cite[Theorem~20]{BFK23}:
	\begin{multline}\label{eq:highest_gen_1}
		(-1)^{\binom{n}{2}} (- \imag \sqrt{w+2})^n \pf_{\chi_{\mathrm{odd}}(n) \le i < j \le n-1} \left(\langle u^i v^j\rangle \left(\frac{\imag (1-u)}{\sqrt{w+2}}\right)^{\chi_{\mathrm{odd}}(n) \delta_{i,1}} \frac{v-u}{1- uv} \right.\\
		\left. \times \left( \frac{-1}{w+2} (1+u)(1 + v) + \frac{ (1-u^2)(1-v^2)}{(1-u)(1-v) - w^2 uv} \right)\right).
	\end{multline}
	
	This conclusion relies on the fact that there are as many $\raisebox{-1.5ex}{\wne}$ as $\raisebox{-1.5ex}{\wsw}$ in any six-vertex model configuration (see \cite[Section~5.1]{BFK23}). In addition, we extract $- \imag \sqrt{w+2}$ from any left boundary vertex, resulting in the weights $\frac{1}{-\imag\sqrt{w+2}}$, $\frac{1}{\imag\sqrt{w+2}}$, $1$ and $1$ for  $\raisebox{-.5ex}{\wlone}$,  $\raisebox{-.5ex}{\wlmone}$,  $\raisebox{-.5ex}{\wlout}$ and  $\raisebox{-.5ex}{\wlin}$, respectively.
	
	On the other hand, \cite[Theorem~20]{BFK23} also implies that
	\begin{equation}\label{eq:highest_weight_alternative}
		\sum_{\SV(n)} w^{\# \raisebox{-1ex}{\scalebox{.75}{\wnw}}+\# \raisebox{-1ex}{\scalebox{.75}{\wse}}}
	\end{equation}
	equals
	\begin{multline}\label{eq:highest_gen_2}
		w^{\binom{n}{2}} \pf_{\chi_{\mathrm{odd}}(n) \le i < j \le n-1} \left(\langle u^i v^j\rangle (1+u)^{\chi_{\mathrm{odd}}(n) \delta_{i,1}} \frac{v-u}{1- uv} \right.\\
		\left. \times \left( (1 - u)(1 - v) + \frac{ w(1-u^2)(1-v^2)}{(w-u)(w-v) - uv} \right)\right)
	\end{multline}
	by setting $p=\frac{1}{w^2}$, $r=\frac{1}{w}$, and $s_+=s_-=1$.
	
	The goal of the remaining part of this proof is to show that \eqref{eq:highest_gen_1} and \eqref{eq:highest_gen_2} conicide. To this end, we use Lemma~\ref{thm:pfaffian_transform}.
	
	Let us first consider the case of $n$ being even. Then \eqref{eq:highest_gen_1} equals
	\begin{multline*}
		(-1)^{\binom{n}{2}} (- \imag \sqrt{w+2})^n\\
		\times \pf_{0 \le i < j \le n-1} \left(\langle u^i v^j\rangle  \frac{(v-u)(1+u)(1+v)((1-u)(1-v)(1+w)+u v w^2)}{(w+2)(1-u v)(1-u-v+(1-w^2)u v)} \right).
	\end{multline*}
	By invoking Lemma~\ref{thm:pfaffian_transform} twice for $f(x,y)=- \frac{(y-x)((1-x)(1-y)(1+w)+x y w^2)}{(1-x y)(1-x-y+(1-w^2)x y)}$, first with $k(x)=\frac{1+x}{-\imag \sqrt{w+2}}$ and $h(x)=1$ and second with $k(x)=1$ and $h(x)=\frac{1}{x-w}$, we deduce that
	\begin{multline*}
		(-1)^{\binom{n}{2}} (- \imag \sqrt{w+2})^n\\
		\times \pf_{0 \le i < j \le n-1} \left(\langle u^i v^j\rangle  \frac{(v-u)(1+u)(1+v)((1-u)(1-v)(1+w)+u v w^2)}{(w+2)(1-u v)(1-u-v+(1-w^2)u v)} \right)\\
		= (-1)^{\binom{n}{2}} \pf_{0 \le i < j \le n-1} \left(\langle u^i v^j\rangle \left(- \frac{(v-u)((1-u)(1-v)(1+w)+u v w^2)}{(1-u v)(1-u-v+(1-w^2)u v)} \right)\right)\\
		= w^{\binom{n}{2}} \pf_{0 \le i < j \le n-1} \left(\langle u^i v^j\rangle \frac{(v-u)(1+u v+w)}{(1-u v)(w- u-v)} \right).
	\end{multline*}
	Moreover, \eqref{eq:highest_gen_2} equals
	\begin{equation*}
		w^{\binom{n}{2}} \pf_{0 \le i < j \le n-1} \left(\langle u^i v^j\rangle   \frac{(v-u)(1-u)(1-v)(1+u v+w)}{(1-u v)(w- u-v)} \right),
	\end{equation*}
	which also equals
	\begin{equation}\label{eq:pfaffian_DSASM_even}
		w^{\binom{n}{2}} \pf_{0 \le i < j \le n-1} \left(\langle u^i v^j\rangle \frac{(v-u)(1+u v+w)}{(1-u v)(w- u-v)} \right)
	\end{equation}
	by applying Lemma~\ref{thm:pfaffian_transform} with $f(x,y)=\frac{(y-x)(1+x y+w)}{(1-x y)(w- x - y)}$, $k(x)=1-x$ and $h(x)=1$.
	
	Notice that by setting $w=1$ in \eqref{eq:pfaffian_DSASM_even}, we obtain the expression for the number of DSASM of even order~$n$ provided in \eqref{eq:DSASMs}.
	
	On the other hand, let $n$ be odd. Then \eqref{eq:highest_gen_1} equals
	\begin{multline}\label{eq:highest_gen_1odd}
		(-1)^{\binom{n}{2}} (- \imag \sqrt{w+2})^{n-1} \pf_{1 \le i < j \le n-1} \left(\langle u^i v^j\rangle  (1-u)^{\delta_{i,1}} \right.\\
		\left. \times  \frac{(v-u)(1+u)(1+v)((1-u)(1-v)(1+w)+u v w^2)}{(2+w)(1-u v)(1-u-v+(1-w^2)u v)} \right),
	\end{multline}
	whereas \eqref{eq:highest_gen_2} equals
	\begin{equation}\label{eq:highest_gen_2odd}
		w^{\binom{n}{2}} \pf_{1 \le i < j \le n-1} \left(\langle u^i v^j\rangle (1+u)^{\delta_{i,1}} \frac{(v-u)(1-u)(1-v)(1+u v+w)}{(1-u v)(w-u-v)} \right).
	\end{equation}
	We will show that both of these expressions are equal to
	\begin{equation}\label{eq:pfaffian_DSASM_odd}
		w^{\binom{n}{2}} \pf_{1 \le i < j \le n-1} \left(\langle u^i v^j\rangle \frac{(v-u)(1+u v+w)}{(1-u v)(w- u-v)} \right)
	\end{equation}
	by applying Lemma~\ref{thm:pfaffian_transform} to \eqref{eq:highest_gen_1odd} and \eqref{eq:highest_gen_2odd}. To this end, we need to remove the prefactors $(1 \pm u)^{\delta_{i,1}}$ and shift the indices of the Pfaffian, which is all accomplished by Lemma~\ref{thm:pfaffian_transform_oddcase}.
	
	We apply Lemma~\ref{thm:pfaffian_transform_oddcase} to \eqref{eq:highest_gen_1odd} and obtain
	\begin{equation*}
		(-1)^{\binom{n}{2}} (- \imag \sqrt{w+2})^{n-1} \pf_{0 \le i < j \le n-2} \left(\langle u^i v^j\rangle  \frac{(v-u)(1+u)(1+v)p_1(w;u,v)}{(2+w)(1-u v)(1-u-v+(1-w^2)u v)} \right),
	\end{equation*}
	where $p_1(w;u,v) \coloneqq (1-u)(1-v)(1+w)+w^2+w^2 (1+w) (1-u v)$. We then apply twice Lemma~\ref{thm:pfaffian_transform}. As in the even case, we first consider $k(x)=\frac{1+x}{- \imag \sqrt{w+2}}$ and $h(x)=1$, which results in
	\begin{equation*}
		(-1)^{\binom{n}{2}} \pf_{0 \le i < j \le n-2} \left(\langle u^i v^j\rangle \frac{-(v-u) p_1(w;u,v)}{(1-u v)(1-u-v+(1-w^2)u v)} \right).
	\end{equation*}
	The second application with $k(x)=1$ and $h(x)=\frac{1}{x-w}$ finally yields
	\begin{equation}\label{eq:highest_gen_odd_equal}
		\underbrace{(-1)^{\binom{n}{2}} (-w)^{\binom{n-1}{2}}}_{=w^{\binom{n-1}{2}}} \pf_{0 \le i < j \le n-2} \left(\langle u^i v^j\rangle  \frac{(v-u)p_2(w;u,v)}{(1-u v)(w-u-v)} \right),
	\end{equation}
	where $p_2(w;u,v) \coloneqq 1+u v+w-2 u w-2 v w+2 w^2-u w^2-v w^2+w^3$.
	
	On the other hand, by applying Lemma~\ref{thm:pfaffian_transform_oddcase} to \eqref{eq:highest_gen_2odd}, we obtain
	\begin{equation*}
		w^{\binom{n}{2}} \pf_{0 \le i < j \le n-2} \left(\langle u^i v^j\rangle  \frac{(v-u)(1-u)(1-v) p_2(w;u,v)}{(w-u)(w-v)(1-u v)(w-u-v)} \right).
	\end{equation*}
	Applying Lemma~\ref{thm:pfaffian_transform} with $k(x)=\frac{1-x}{w-x}$ and $h(x)=1$ to the previous expression then yields the following which coincides with \eqref{eq:highest_gen_odd_equal}:
	\begin{equation*}
		\underbrace{w^{\binom{n}{2}} w^{-n+1}}_{=w^{\binom{n-1}{2}}} \pf_{0 \le i < j \le n-2} \left(\langle u^i v^j\rangle  \frac{(v-u) p_2(w;u,v)}{(1-u v)(w-u-v)} \right).
	\end{equation*}
	This proves the equality of the generating functions \eqref{eq:highest_gen_1odd} and \eqref{eq:highest_gen_2odd}.
	
	Finally, we want to prove that \eqref{eq:highest_gen_2odd} coincides with \eqref{eq:pfaffian_DSASM_odd}. Therefore, let $f(u,v) = \sum_{i,j \ge 0} f_{i,j} u^i v^j$ again be an antisymmetric power series in $u$ and $v$. Consider the two triangular arrays
	\begin{equation*}
		\left(\langle u^i v^j\rangle f(u,v) \right)_{1 \le i < j \le n-1}\quad\text{and}\quad \left(\langle u^i v^j\rangle (1+u)^{\delta_{i,1}} (1-u)(1-v) f(u,v) \right)_{1 \le i < j \le n-1}.
	\end{equation*}
	Let $A$ and $B$, respectively, denote their skew-symmetric completions, which are, by definition, $(n-1) \times (n-1)$ skew-symmetric matrices. We claim that $A$ is obtained from $B$ by sequentially performing the following elementary row and column operations on $B$ for $i=1,\dots,n-2$: add the $i$th row to the $(i+1)$th row and the $i$th column to the $(i+1)$th column. Note that this observation would prove the equality of $\eqref{eq:highest_gen_2odd}$ and $\eqref{eq:pfaffian_DSASM_odd}$ by taking $f(u,v)=\frac{(v-u)(1+u v+w)}{(1-u v)(w- u-v)}$.
	
	To show this observation, let $M$ be the skew-symmetric completion of the triangular array $( m_{i,j} )_{1 \le i < j \le n-1}$. It is straightforward to see that the aforementioned elementary row and column operations on $M$ result in the skew-symmetric completion of $\left( \sum_{k=1}^{i} \sum_{l=i+1}^{j} m_{k,l} \right)_{1 \le i < j \le  n-1}$.
	
	Furthermore, let $F_{i,j} \coloneqq \langle u^i v^j\rangle (1+u)^{\delta_{i,1}}(1-u)(1-v) f(u,v)$. We see that
	\begin{equation*}
		F_{i,j} =
		\begin{cases}
			f_{1,j} - f_{1,j-1},&1 = i < j,\\
			f_{i,j} - f_{i-1,j} - f_{i,j-1} + f_{i-1,j-1},&1 \neq i < j.
		\end{cases}
	\end{equation*}
	Thus, we need to show that
	\begin{equation*}
		\left( \sum_{k=1}^{i} \sum_{l=i+1}^{j} F_{k,l} \right)_{1 \le i < n-1} = f_{i,j}.
	\end{equation*}
	The proof will be accomplished by induction on $i$. For the base case, we obtain by telescoping
	\begin{equation*}
		\sum_{l=2}^{j} F_{1,l} = \sum_{l=2}^{j} \left( f_{1,l} - f_{1,l-1} \right) = f_{1,j}.
	\end{equation*}
	Note that $f_{i,i}=0$ for all~$i$ due to the antisymmetry of $f$.
	
	For the induction step, assume $i > 1$. First, we observe that
	\begin{equation*}
		\sum_{l=i+1}^{j} F_{i,l} = f_{i,j} - f_{i-1,j} + f_{i-1,i}
	\end{equation*}
	and
	\begin{equation*}
		\sum_{k=1}^{i-1} F_{k,i} = f_{i-1,i}.
	\end{equation*}
	Then, by using the induction hypothesis, we argue that
	\begin{equation*}
		\sum_{k=1}^{i} \sum_{l=i+1}^{j} F_{k,l}
		= \sum_{k=1}^{i-1} \sum_{l=i}^{j} F_{k,l} + \sum_{l=i+1}^{j} F_{i,l} - \sum_{k=1}^{i-1} F_{k,i}
		= f_{i-1,j} +f_{i,j} - f_{i-1,j} + f_{i-1,i} - f_{i-1,i}
		= f_{i,j},
	\end{equation*}
	which concludes the proof.
\end{proof}

The proof of Theorem~\ref{thm:highest_term} equates the two generating functions \eqref{eq:highest_weight} and \eqref{eq:highest_weight_alternative}. This raises the question of finding a bijective proof. We propose the following problem; see Table~\ref{tab:bijection} for an illustration of the  problem for $n=3$.

\begin{prob}\label{prob:bijective_proof}
	Find a bijection between DSASMs of order~$n$ to prove the following identity:
	\begin{equation*}
		(-1)^{\binom{n+1}{2}} \sum_{\SV(n)}  (-1)^{\# \scalebox{.75}{\wlone}+\# \scalebox{.75}{\wlin}}  w^{\# \raisebox{-1ex}{\scalebox{.75}{\wne}}+\# \raisebox{-1ex}{\scalebox{.75}{\wsw}}} (w+2)^{\# \scalebox{.75}{\wlout}} = \sum_{\SV(n)} w^{\# \raisebox{-1ex}{\scalebox{.75}{\wnw}}+\# \raisebox{-1ex}{\scalebox{.75}{\wse}}}.
	\end{equation*}
\end{prob}

Note that the left-hand side of the previous equation follows from the generating function \eqref{eq:highest_weight} if we multiply the weight of every left boundary vertex in Table~\ref{tab:six-vertex weights highest term} by $-1$ and bear in mind that there are as many $\raisebox{-.5ex}{\wlout}$ as $\raisebox{-.5ex}{\wlin}$ in any six-vertex model configuration as we have seen in Section~\ref{sec:six-vertex}.

\begin{table}[htbp]
	\centering
	\resizebox{\textwidth}{!}{\begin{tabular}{lccccc}\toprule
			DSASM&
			$\begin{pmatrix}
				1 & 0 & 0\\
				0 & 1 & 0\\
				0 & 0 & 1
			\end{pmatrix}$&
			$\begin{pmatrix}
				1 & 0 & 0\\
				0 & 0 & 1\\
				0 & 1 & 0
			\end{pmatrix}$&
			$\begin{pmatrix}
				0 & 1 & 0\\
				1 & 0 & 0\\
				0 & 0 & 1
			\end{pmatrix}$&
			$\begin{pmatrix}
				0 & 0 & 1\\
				0 & 1 & 0\\
				1 & 0 & 0
			\end{pmatrix}$&
			$\begin{pmatrix}
				0 & 1 & 0\\
				1 &-1 & 1\\
				0 & 1 & 0
			\end{pmatrix}$\\\midrule
			$\SV$&
			\begin{tikzpicture}[scale=.35,baseline=(current bounding box.center),every node/.style={sloped,allow upside down}]
				
				\draw[very thick] (0,0) --node{\midarrow} (0,2);
				\draw[very thick] (2,0) --node{\midarrow} (2,2);
				\draw[very thick] (4,0) --node{\midarrow} (4,2);
				
				\draw[very thick] (6,0) --node{\midarrow} (4,0);
				\draw[very thick] (6,-2) --node{\midarrow} (4,-2);
				\draw[very thick] (6,-4) --node{\midarrow} (4,-4);
				
				\draw[very thick] (2,0) --node{\midarrow} (0,0);
				
				\draw[very thick] (2,-2) --node{\midarrow} (2,0);
				\draw[very thick] (4,-2) --node{\midarrow} (2,-2);
				
				\draw[very thick] (4,-4) --node{\midarrow} (4,-2);
				
				\draw[very thick] (4,0) --node{\midarrow} (2,0);
				\draw[very thick] (4,-2) --node{\midarrow} (4,0);
				
				\node at (0,2) {$\bullet$};
				\node at (2,2) {$\bullet$};
				\node at (4,2) {$\bullet$};
				
				\node at (0,0) {\color{orange}{$\bullet$}};
				\node at (2,0) {\color{teal}{$\bullet$}};
				\node at (4,0) {\color{teal}{$\bullet$}};
				
				\node at (2,-2) {\color{orange}{$\bullet$}};
				\node at (4,-2) {\color{teal}{$\bullet$}};
				
				\node at (4,-4) {\color{orange}{$\bullet$}};
				
				\node at (6,0) {$\bullet$};
				\node at (6,-2) {$\bullet$};
				\node at (6,-4) {$\bullet$};
			\end{tikzpicture}&
			\begin{tikzpicture}[scale=.35,baseline=(current bounding box.center),every node/.style={sloped,allow upside down}]
				
				\draw[very thick] (0,0) --node{\midarrow} (0,2);
				\draw[very thick] (2,0) --node{\midarrow} (2,2);
				\draw[very thick] (4,0) --node{\midarrow} (4,2);
				
				\draw[very thick] (6,0) --node{\midarrow} (4,0);
				\draw[very thick] (6,-2) --node{\midarrow} (4,-2);
				\draw[very thick] (6,-4) --node{\midarrow} (4,-4);
				
				\draw[very thick] (2,0) --node{\midarrow} (0,0);
				
				\draw[very thick] (2,-2) --node{\midarrow} (2,0);
				\draw[very thick] (2,-2) --node{\midarrow} (4,-2);
				
				\draw[very thick] (4,-2) --node{\midarrow} (4,-4);
				
				\draw[very thick] (4,0) --node{\midarrow} (2,0);
				\draw[very thick] (4,-2) --node{\midarrow} (4,0);
				
				\node at (0,2) {$\bullet$};
				\node at (2,2) {$\bullet$};
				\node at (4,2) {$\bullet$};
				
				\node at (0,0) {\color{orange}{$\bullet$}};
				\node at (2,0) {\color{teal}{$\bullet$}};
				\node at (4,0) {\color{teal}{$\bullet$}};
				
				\node at (2,-2) {\color{violet}{$\bullet$}};
				\node at (4,-2) {$\bullet$};
				
				\node at (4,-4) {\color{orange}{$\bullet$}};
				
				\node at (6,0) {$\bullet$};
				\node at (6,-2) {$\bullet$};
				\node at (6,-4) {$\bullet$};
			\end{tikzpicture}&
			\begin{tikzpicture}[scale=.35,baseline=(current bounding box.center),every node/.style={sloped,allow upside down}]
				
				\draw[very thick] (0,0) --node{\midarrow} (0,2);
				\draw[very thick] (2,0) --node{\midarrow} (2,2);
				\draw[very thick] (4,0) --node{\midarrow} (4,2);
				
				\draw[very thick] (6,0) --node{\midarrow} (4,0);
				\draw[very thick] (6,-2) --node{\midarrow} (4,-2);
				\draw[very thick] (6,-4) --node{\midarrow} (4,-4);
				
				\draw[very thick] (0,0) --node{\midarrow} (2,0);
				
				\draw[very thick] (2,0) --node{\midarrow} (2,-2);
				\draw[very thick] (4,-2) --node{\midarrow} (2,-2);
				
				\draw[very thick] (4,-4) --node{\midarrow} (4,-2);
				
				\draw[very thick] (4,0) --node{\midarrow} (2,0);
				\draw[very thick] (4,-2) --node{\midarrow} (4,0);	
				
				\node at (0,2) {$\bullet$};
				\node at (2,2) {$\bullet$};
				\node at (4,2) {$\bullet$};
				
				\node at (0,0) {\color{violet}{$\bullet$}};
				\node at (2,0) {$\bullet$};
				\node at (4,0) {\color{teal}{$\bullet$}};
				
				\node at (2,-2) {\color{orange}{$\bullet$}};
				\node at (4,-2) {\color{teal}{$\bullet$}};
				
				\node at (4,-4) {\color{orange}{$\bullet$}};
				
				\node at (6,0) {$\bullet$};
				\node at (6,-2) {$\bullet$};
				\node at (6,-4) {$\bullet$};
			\end{tikzpicture}&
			\begin{tikzpicture}[scale=.35,baseline=(current bounding box.center),every node/.style={sloped,allow upside down}]
				
				\draw[very thick] (0,0) --node{\midarrow} (0,2);
				\draw[very thick] (2,0) --node{\midarrow} (2,2);
				\draw[very thick] (4,0) --node{\midarrow} (4,2);
				
				\draw[very thick] (6,0) --node{\midarrow} (4,0);
				\draw[very thick] (6,-2) --node{\midarrow} (4,-2);
				\draw[very thick] (6,-4) --node{\midarrow} (4,-4);
				
				\draw[very thick] (0,0) --node{\midarrow} (2,0);
				
				\draw[very thick] (2,-2) --node{\midarrow} (2,0);
				\draw[very thick] (4,-2) --node{\midarrow} (2,-2);
				
				\draw[very thick] (4,-2) --node{\midarrow} (4,-4);
				
				\draw[very thick] (2,0) --node{\midarrow} (4,0);
				\draw[very thick] (4,0) --node{\midarrow} (4,-2);
				
				\node at (0,2) {$\bullet$};
				\node at (2,2) {$\bullet$};
				\node at (4,2) {$\bullet$};
				
				\node at (0,0) {\color{violet}{$\bullet$}};
				\node at (2,0) {\color{cyan}{$\bullet$}};
				\node at (4,0) {$\bullet$};
				
				\node at (2,-2) {\color{orange}{$\bullet$}};
				\node at (4,-2) {\color{cyan}{$\bullet$}};
				
				\node at (4,-4) {\color{orange}{$\bullet$}};
				
				\node at (6,0) {$\bullet$};
				\node at (6,-2) {$\bullet$};
				\node at (6,-4) {$\bullet$};
			\end{tikzpicture}&
			\begin{tikzpicture}[scale=.35,baseline=(current bounding box.center),every node/.style={sloped,allow upside down}]
				
				\draw[very thick] (0,0) --node{\midarrow} (0,2);
				\draw[very thick] (2,0) --node{\midarrow} (2,2);
				\draw[very thick] (4,0) --node{\midarrow} (4,2);
				
				\draw[very thick] (6,0) --node{\midarrow} (4,0);
				\draw[very thick] (6,-2) --node{\midarrow} (4,-2);
				\draw[very thick] (6,-4) --node{\midarrow} (4,-4);
				
				\draw[very thick] (0,0) --node{\midarrow} (2,0);
				
				\draw[very thick] (2,0) --node{\midarrow} (2,-2);
				\draw[very thick] (2,-2) --node{\midarrow} (4,-2);
				
				\draw[very thick] (4,-2) --node{\midarrow} (4,-4);
				
				\draw[very thick] (4,0) --node{\midarrow} (2,0);
				\draw[very thick] (4,-2) --node{\midarrow} (4,0);
				
				\node at (0,2) {$\bullet$};
				\node at (2,2) {$\bullet$};
				\node at (4,2) {$\bullet$};
				
				\node at (0,0) {\color{violet}{$\bullet$}};
				\node at (2,0) {$\bullet$};
				\node at (4,0) {\color{teal}{$\bullet$}};;
				
				\node at (2,-2) {$\bullet$};
				\node at (4,-2) {$\bullet$};
				
				\node at (4,-4) {\color{orange}{$\bullet$}};
				
				\node at (6,0) {$\bullet$};
				\node at (6,-2) {$\bullet$};
				\node at (6,-4) {$\bullet$};
			\end{tikzpicture}\\\midrule
			\makecell{$(-1)^{\# \scalebox{.75}{\wloneC}+\# \scalebox{.75}{\wlinC}} w^{\# \raisebox{-1ex}{\scalebox{.75}{\wneC}}+\# \raisebox{-1ex}{\scalebox{.75}{\wswC}}}$\\
				$\times(w+2)^{\# \scalebox{.75}{\wloutC}}$}&
			$-1$&$w+2$&$w+2$&$w^2(w+2)$&$-(w+2)$\\\midrule
			$w^{\# \raisebox{-1ex}{\scalebox{.75}{\wnwC}}+\# \raisebox{-1ex}{\scalebox{.75}{\wseC}}}$&
			$w^3$&$w^2$&$w^2$&$1$&$w$\\\bottomrule
	\end{tabular}}
	\caption{An illustration of Problem~\ref{prob:bijective_proof} for $n=3$. We list all five DSASMs of order~$3$, their corresponding six-vertex model configurations and the weights related to the two generating functions in consideration (note the colour coding). In both cases, the sum of weights is $1+w+2 w^2+w^3$.}
	\label{tab:bijection}
\end{table}

\section{How to recover the classical Littlewood identity}
\label{sec:recovery}

In this section, we show how to recover the classical Littlewood identity~\eqref{eq:Littlewood} from our generalization in Theorem~\ref{thm:Robbins_Littlewood}. We multiply both sides of \eqref{eq:Robbins_Littlewood} by $w^{-\binom{n}{2}}$ and let $w\to\infty$ afterwards. For the left-hand side, consider DAMTs with weight function \eqref{eq:weight_DAMT}. Setting $u=v=1$ and multiplying by $w^{-\binom{n}{2}}$ changes the weight to
\[
w^{-\# \searrow - \# \swarrow}
\prod_{i=1}^n x_i^{\sum_{j=1}^{i} a_{i,j}-\sum_{j=1}^{i-1} a_{i-1,j} + \# \searrow \text{ in row~$(i-1)$ } - \# \swarrow \text{ in row~$(i-1)$}}.
\]
Letting $w\to\infty$ implies that only DAMTs with central entries above the bottom row remain. Consequently, the resulting DAMTs have strict increase along $\nearrow$-diagonals and along $\searrow$-diagonals. These DAMTs are in bijective correspondence with Gelfand--Tsetlin patterns, which can be seen as follows: For all $1\le i \le n$, we first subtract $i-1$ from the $i$th $\searrow$-diagonal and then $i-1$ from the $i$th $\swarrow$-diagonal; all diagonals are numbered from left to right. Since Gelfand--Tsetlin patterns are enumerated by Schur polynomials, the left-hand side of \eqref{eq:Robbins_Littlewood} reduces to $\prod_{i=1}^{n} x_i^{n-1} \sum_{\lambda} s_{\lambda} (x_1,\dots,x_n)$. (Note that these considerations now establish a combinatorial proof of Proposition~\ref{prop:RobbinsToSchur}~\ref{prop:RobbinsToSchur1} that is based on the combinatorial model for modified Robbins polynomials provided in Theorem~\ref{thm:gfun}.)

\begin{exm}
	We provide an example of the bijection between DAMTs with only central entries and Gelfand--Tsetlin patterns:
	\begin{equation*}
		\begin{tikzpicture}[baseline=(current bounding box.center),scale=0.65]
			\node at (0,0) {$1$};
			\node at (2,0) {$5$};
			\node at (4,0) {$8$};
			\node at (6,0) {$10$};
			\node at (1,1) {$2$};
			\node at (3,1) {$6$};
			\node at (5,1) {$9$};
			\node at (2,2) {$4$};
			\node at (4,2) {$7$};
			\node at (3,3) {$5$};
			
			\node at (1,.45) {\footnotesize $\downarrow$};
			\node at (3,.45) {\footnotesize $\downarrow$};
			\node at (5,.45) {\footnotesize $\downarrow$};
			\node at (2,1.45) {\footnotesize $\downarrow$};
			\node at (4,1.45) {\footnotesize $\downarrow$};
			\node at (3,2.45) {\footnotesize $\downarrow$};
		\end{tikzpicture}\longleftrightarrow
		\begin{tikzpicture}[baseline=(current bounding box.center),scale=0.65]
			\node at (0,0) {$1$};
			\node at (2,0) {$3$};
			\node at (4,0) {$4$};
			\node at (6,0) {$4$};
			\node at (1,1) {$1$};
			\node at (3,1) {$3$};
			\node at (5,1) {$4$};
			\node at (2,2) {$2$};
			\node at (4,2) {$3$};
			\node at (3,3) {$2$};
		\end{tikzpicture}
	\end{equation*}
\end{exm}

For the right-hand side of \eqref{eq:Robbins_Littlewood}, consider the six-vertex model configurations. Multiplying with $w^{-\binom{n}{2}}$ results in dividing the weight of every bulk vertex by $w$. Taking the limit $w\to\infty$ then excludes the local configurations $\raisebox{-1.5ex}{\wnw}$ and $\raisebox{-1.5ex}{\wse}$, since the respective weights tend to zero. The only six-vertex model configuration that avoids these two local configurations is the one that corresponds to the DSASM with $1$s along the antidiagonal and $0$s everywhere else, the weight of which is $\prod_{i=1}^{n} x_i^{n-1}$. An alternative way to obtain the limit of the right-hand side is to use the corresponding identity in \eqref{eq:Robbins_Littlewood} together with the following version of Schur's Pfaffian identity\footnote{This identity follows from the classical Schur's Pfaffian identity $\pf_{1 \le i < j \le 2n} \left( \frac{x_j-x_i}{x_i + x_j} \right) = \prod_{1 \le i < j \le 2n} \frac{x_j-x_i}{x_i + x_j}$ after applying the transformation $x_i \mapsto \frac{1+x_i}{1-x_i}$ for all $1 \le i \le 2n$. For the odd case, we additionally set $x_{2n}=1$ and rearrange appropriately.} \cite{Sch11}:
\[
\pf_{\chi_{\mathrm{even}}(n) \le i < j \le n} \left(
\begin{cases}
	1, & i=0,\\
	\frac{x_j-x_i}{1-x_i x_j},& i\ge 1
\end{cases} \right)
= \prod_{1 \le i < j \le n} \frac{x_j-x_i}{1- x_i x_j}.
\]		
After dividing both sides by $\prod_{i=1}^{n} x_i^{n-1}$, we eventually obtain \eqref{eq:Littlewood}.

Another way to recover the classical Littlewood identity is by setting $w=0$, in which case Proposition~\ref{prop:RobbinsToSchur}~\ref{prop:RobbinsToSchur2} implies that the left-hand side of \eqref{eq:Robbins_Littlewood} reduces to
\begin{multline*}
	\sum_{0 \le k_1 < \dots < k_n} s_{(k_n-n+1,k_{n-1}-n+2,\dots,k_1)}(x_{1},\ldots,x_{n}) \prod_{1 \le i < j \le n} (1+x_i x_j)\\
	= 
	\sum_{\lambda} s_{\lambda} (x_1,\dots,x_n) \prod_{1 \le i < j \le n} (1+x_i x_j).
\end{multline*}

On the other hand, the right-hand side of \eqref{eq:Robbins_Littlewood} becomes
\begin{equation*}
	\prod_{i=1}^n \frac{1}{1-x_i} \prod_{1 \le i < j \le n} \frac{1+x_i x_j}{1-x_i x_j}
\end{equation*}
due to the following transformation of Schur's Pfaffian identity\footnote{In this case, we apply the transformation $x_i \mapsto \frac{x_i}{x_i^2-1}$ for all $1 \le i \le 2n$ to the classical Schur's Pfaffian identity.}:
\begin{equation*}
	\pf_{\chi_{\mathrm{even}}(n) \le i < j \le n} \left(
	\begin{cases}
		1, & i=0,\\
		\frac{(x_j-x_i)(1+x_i x_j)}{(x_i+x_j)(1-x_i x_j)},& i\ge 1
	\end{cases} \right)
	= \prod_{1 \le i < j \le n} \frac{(x_j-x_i)(1+x_i x_j)}{(x_i+x_j)(1-x_i x_j)}.
\end{equation*}
Dividing both sides by $\prod_{1 \le i < j \le n} (1+x_i x_j)$ yields the classical Littlewood identity~\eqref{eq:Littlewood}.

The question is whether there are also combinatorial means to establish both sides. As for the left-hand side, 
\begin{equation*}
	R_{(k_{1},\ldots,k_{n})}^{\ast}(x_{1},\ldots,x_{n};1,1,0) = s_{(k_n-n+1,k_{n-1}-n+2,\dots,k_1)}(x_{1},\ldots,x_{n}) \prod\limits_{1 \le i < j \le n} (1+x_i x_j)
\end{equation*}
follows using Theorem~\ref{thm:gfun} and urban renewal as indicated in \cite{Fis23}. Regarding the right-hand side, we need to show 
\begin{equation}
	\label{eq:2enum} 
	\left. Z_{\DSASM}(x_1,\ldots,x_n) \right\vert_{w=0} = \prod_{1 \le i < j \le n} (1+x_i x_j).
\end{equation} 
We modify the weights in Table~\ref{tab:six-vertex weights} as follows: first we multiply the left boundary weights with $-\tau(x)$, which results in 
total in an extra factor $(-1)^n \prod_{i=1}^n \tau(x_i)$ of the partition function. By using the fact that, along $\raisebox{-.5ex}{\wL}$-shaped paths, there is one more occurrence of $1$ than of $-1$, and by further multiplying by $\prod_{i=1}^n \tau(x_i)$, we can ensure that the weights of the local configurations corresponding 
to the matrix entries $-1$ are equal to $1$. Finally, we can suppress the minus signs in the left boundary weights for $\raisebox{-.25ex}{\wlout}$ and 
$\raisebox{-.25ex}{\wlin}$ as there is the same number of such local configurations, see the proof of Proposition~\ref{prop:polynomial}. After setting $w=0$, we 
obtain the weights provided in Table~\ref{tab:six-vertex weightsw=0}.

\begin{table}[htbp]
	\centering
	\begin{tabular}{llc}\toprule
		bulk weights & left boundary weights & matrix entries\\\midrule\medskip
		$W\Big(\raisebox{-1.5ex}{\wone};x,y\Big)= (1+x^2)(1+y^2)$ & $W\Big(\raisebox{-.25ex}{\wlone};x\Big)=1+x^2$ & $1$\\\medskip
		$W\Big(\raisebox{-1.5ex}{\wmone};x,y\Big)=1$ & $W\Big(\raisebox{-.25ex}{\wlmone};x\Big)= 1$ & $-1$\\\medskip
		$W\Big(\raisebox{-1.5ex}{\wne};x,y\Big) = x + y $ & $W\Big(\raisebox{-.25ex}{\wlout};x\Big)=\sqrt{2}x$ & $0$\\\medskip
		$W\Big(\raisebox{-1.5ex}{\wsw};x,y\Big) = x + y$ & $W\Big(\raisebox{-.25ex}{\wlin};x\Big)=\sqrt{2}x$ & $0$\\\medskip
		$W\Big(\raisebox{-1.5ex}{\wnw};x,y\Big) = 1-x y$ & & $0$\\\medskip
		$W\Big(\raisebox{-1.5ex}{\wse};x,y\Big) = 1-x y$ & & $0$\\\bottomrule
	\end{tabular}
	\caption{Modified vertex weights in the six-vertex model for $w=0$.}
	\label{tab:six-vertex weightsw=0}
\end{table}

The following is based on the relation between the $2$-enumeration of ASMs and the tilings of the Aztec diamond, see \cite{Elk92a,Elk92b} and \cite[Remark 4.3]{Ciucu97}: Building on the considerations presented in \cite[Section 5.7]{ABF20}, the partition function $Z_{\DSASM}(x_1,\allowbreak \ldots,x_n)$ with respect to the weights in Table~\ref{tab:six-vertex weightsw=0} is the weighted enumeration of the matchings of the edge-weighted graph sketched in Figure~\ref{fig:matching}, where all vertices are covered except for possibly some on the left boundary.

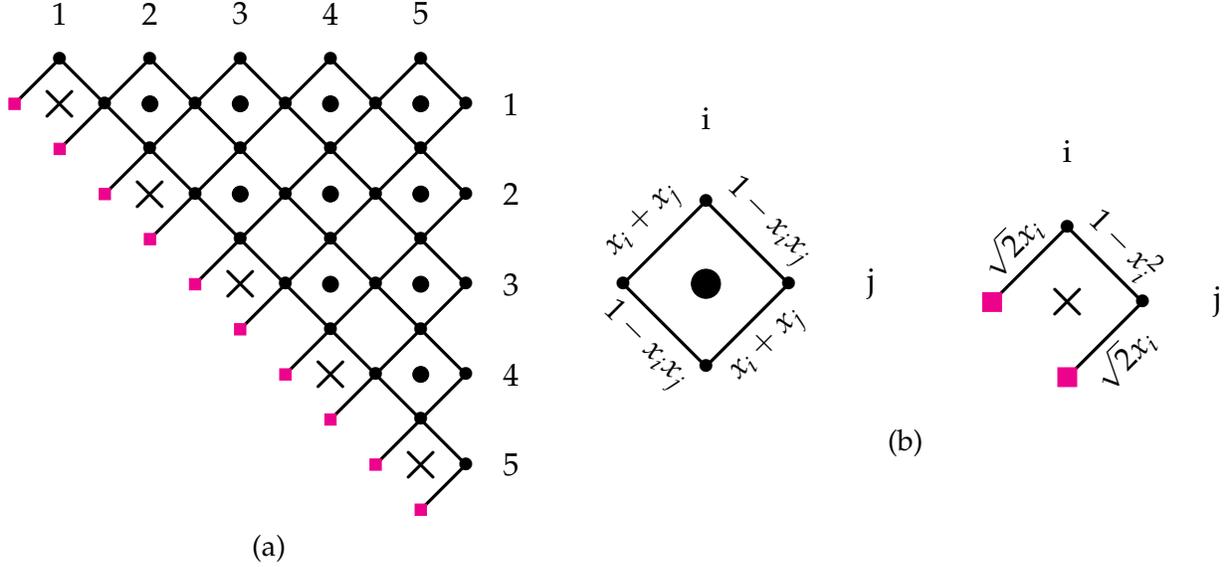
\begin{figure}[htbp]
	\centering
	\begin{subfigure}[t]{.5\textwidth}
		\centering
		\begin{tikzpicture}[scale=.25]
			\foreach \x in {2,6,...,20}
			{\node at (\x,2) {$\bullet$};}
			\foreach \x in {4,8,...,20}
			{\node at (\x,0) {$\bullet$};}
			\foreach \x in {6,10,...,20}
			{\node at (\x,-2) {$\bullet$};}
			\foreach \x in {8,12,...,20}
			{\node at (\x,-4) {$\bullet$};}
			\foreach \x in {10,14,18}
			{\node at (\x,-6) {$\bullet$};}
			\foreach \x in {12,16,20}
			{\node at (\x,-8) {$\bullet$};}
			\foreach \x in {14,18}
			{\node at (\x,-10) {$\bullet$};}
			\foreach \x in {16,20}
			{\node at (\x,-12) {$\bullet$};}
			\foreach \x in {18}
			{\node at (\x,-14) {$\bullet$};}
			\node at (20,-16) {$\bullet$};
			
			\draw[very thick] (0,0) -- (2,2) -- (20,-16) -- (18,-18);
			\draw[very thick] (2,-2) -- (6,2) -- (20,-12) -- (16,-16);
			\draw[very thick] (4,-4) -- (10,2) -- (20,-8) -- (14,-14);
			\draw[very thick] (6,-6) -- (14,2) -- (20,-4) -- (12,-12);
			\draw[very thick] (8,-8) -- (18,2) -- (20,0) -- (10,-10);
			
			\foreach \x in {0,2,...,18}
			{\draw[color=magenta,fill]  (\x-.25,-\x-.25) rectangle ++(.5,.5);}
			
			\filldraw (6,0) circle (10pt);
			\filldraw (10,0) circle (10pt);
			\filldraw (14,0) circle (10pt);
			\filldraw (18,0) circle (10pt);
			
			\filldraw (10,-4) circle (10pt);
			\filldraw (14,-4) circle (10pt);
			\filldraw (18,-4) circle (10pt);
			
			\filldraw (14,-8) circle (10pt);
			\filldraw (18,-8) circle (10pt);
			
			\filldraw (18,-12) circle (10pt);
			
			\node at (2,0) {\scalebox{1.5}{$\boldsymbol{\times}$}};
			\node at (6,-4) {\scalebox{1.5}{$\boldsymbol{\times}$}};
			\node at (10,-8) {\scalebox{1.5}{$\boldsymbol{\times}$}};
			\node at (14,-12) {\scalebox{1.5}{$\boldsymbol{\times}$}};
			\node at (18,-16) {\scalebox{1.5}{$\boldsymbol{\times}$}};
			
			\node at (2,4) {$1$};
			\node at (6,4) {$2$};
			\node at (10,4) {$3$};
			\node at (14,4) {$4$};
			\node at (18,4) {$5$};
			
			\node at (22,0) {$1$};
			\node at (22,-4) {$2$};
			\node at (22,-8) {$3$};
			\node at (22,-12) {$4$};
			\node at (22,-16) {$5$};
		\end{tikzpicture}
		\caption{}
		\label{subfig:graph}
	\end{subfigure}
	\hfill
	\begin{subfigure}[t]{.45\textwidth}
		\centering
		\begin{tikzpicture}[scale=.5]
			\node at (0,8) {$\bullet$};
			\node at (2,10) {$\bullet$};
			\node at (4,8) {$\bullet$};
			\node at (2,6) {$\bullet$};
			
			\draw[very thick] (0,8) -- (2,10) -- (4,8) -- (2,6) -- cycle;
			
			\node[rotate=45] at (.5,9.5) {$x_i+x_j$};
			\node[rotate=-45] at (3.5,9.5) {$1-x_i x_j$};
			\node[rotate=45] at (3.5,6.5) {$x_i+x_j$};
			\node[rotate=-45] at (.5,6.5) {$1-x_i x_j$};
			
			\filldraw (2,8) circle (10pt);
			
			\node at (2,11.5) {$\mathrm{i}$};
			\node at (5.5,8) {$\mathrm{j}$};
			
			\node at (2,2) {$\bullet$};
			\node at (4,0) {$\bullet$};
			
			\draw[very thick] (0,0) -- (2,2) -- (4,0) -- (2,-2);
			
			\draw[color=magenta,fill]  (-.25,-.25) rectangle ++(.5,.5);
			\draw[color=magenta,fill]  (1.75,-2.25) rectangle ++(.5,.5);
			
			\node[rotate=45] at (.5,1.5) {$\sqrt{2}x_i$};
			\node[rotate=-45] at (3.5,1.5) {$1-x_i^2$};
			\node[rotate=45] at (3.5,-1.5) {$\sqrt{2}x_i$};
			
			\node at (2,0) {\scalebox{1.5}{$\boldsymbol{\times}$}};
			
			\node at (2,3.5) {$\mathrm{i}$};
			\node at (5.5,0) {$\mathrm{j}$};
		\end{tikzpicture}
		\caption{}
		\label{subfig:weights}
	\end{subfigure}
	\caption{The partition function~$Z_{\DSASM}(x_1,\ldots,x_n)$ with respect to the weights in Table~\ref{tab:six-vertex weightsw=0} is given by the generating function of perfect matchings of the grid graph in \subref{subfig:graph} (here for $n=5$), where the red vertices on the left boundary may remain uncovered. The corresponding edge weights are provided in \subref{subfig:weights}.}
	\label{fig:matching}
\end{figure}

We pose it as an open problem to provide a combinatorial proof that this weighted enumeration is $\prod_{1 \le i \le j \le n} (1+x_i x_j)$, where the additional factor $\prod_{i=1}^n (1 + x_i^2)$ appears compared to \eqref{eq:2enum} due to the modification of weights. A natural question is whether \deff{urban renewal} can be applied here as in the case of the Aztec diamond, see for instance \cite{Propp23}.

\section{A conjectural generalization of Theorem~\ref{thm:Robbins_Littlewood} for fully inhomogeneous spin Hall--Littlewood polynomials}\label{sec:conjecture}

In this section, we present a conjectural Littlewood identity for fully inhomogeneous spin Hall--Littlewood polynomials that generalizes Theorem~\ref{thm:Robbins_Littlewood}, along with possible consequences including a new Littlewood identity for Hall--Littlewood polynomials. It is planned to study the identity for fully inhomogeneous spin Hall--Littlewood polynomials in a forthcoming paper.

Recall that, for a partition $\lambda$, $m_r(\lambda)$ denotes the number of parts equal to $r$.

\begin{conj} 
	\label{conj:fully} 
	Let $n$ be a positive integer and $p$ a non-negative integer. Setting  
	$\xi_j=1$ for all $j$ in $\F_{\lambda}(u_1,\ldots,u_n)$, we conjecture
	\begin{multline} 
		\label{conjid:fully}
		\sum_{\lambda_1 \ge \lambda_2 \ge \ldots \ge \lambda_n \ge 0} 
		\prod_{r \ge 0} \frac{(-s_r;q^{\half})_{m_r(\lambda)}}{(q^{\half};q^{\half})_{m_r(\lambda)}} \F_{\lambda}(u_1,\ldots,u_n)
		=  \prod_{i=1}^{n} \frac{1+q^{\half}}{1- u_i}
		\prod_{1 \le i < j \le n} \frac{1-q  u_i u_j}{u_i-u_j} \\
		\times \pf_{\chi_{\mathrm{even}}(n) \le i < j \le n} \left(
		\begin{cases}
			1, & i=0,\\
			\frac{(u_i-u_j)((1+q)(1- q^{\half} u_i u_j) + (u_i + u_j)(q^{\half}- q))}{(1+q^{\half})(1- u_i u_j)(1- q u_i u_j)},& i\ge 1,
		\end{cases} \right) 
	\end{multline}
	where $s_j=s$ for all $j \ge p$ and $u_i = \frac{s+x_i}{1 + s x_i}$ for all $i$, and both sides are considered as power series in $x_1,x_2,\ldots,x_n$.
\end{conj} 
The conjecture has been checked for $n=1$ and arbitrary~$p$, for $n=2$ if $p=0,1$, and for $n=3,4$ if $p=0$. 

Next, we demonstrate that the case 
$s_j=-q^{-\half}$ is equivalent to Theorem~\ref{thm:Robbins_Littlewood}. Indeed, it holds
\begin{multline*} 
	\sum_{\lambda_1 \ge \lambda_2 \ge \ldots \ge \lambda_n \ge 0} 
	\prod_{r \ge 0} \frac{(q^{-\half};q^{\half})_{m_r(\lambda)}}{(q^{\half};q^{\half})_{m_r(\lambda)}} F_{\lambda}(u_1,\ldots,u_n) \\ 
	\shoveleft{ =   q^{\binom{n}{2}/2} \prod_{i=1}^n (q^{\half}-x_i)}\\
	\left.\times\sum_{0 \le \lambda_n < \lambda_{n-1} < \ldots < \lambda_1} 
	R^{\ast}_{(\lambda_n,\ldots,\lambda_1)}(x_1,\ldots,x_n;1,1,-q^{\half}-q^{-\half})
	\right|_{x_i = \frac{u_i+q^{-\half}}{1+q^{-\half} u_i}}\\
	\shoveleft{ = \prod_{i=1}^n \frac{1+q^{\half}}{1-u_i} \prod_{1 \le i < j \le n} \frac{1-q u_i u_j}{u_i-u_j}}\\
	\times \pf_{\chi_{\mathrm{even}}(n) \le i < j \le n} \left(
	\begin{cases}
		1, & i=0,\\
		\frac{(u_i-u_j)((1+q)(1- q^{\half} u_i u_j) + (u_i + u_j)(q^{\half}- q))}{(1+q^{\half})(1- u_i u_j)(1- q u_i u_j)},& i\ge 1,
	\end{cases} \right)
\end{multline*} 
where we have used Theorem~\ref{thm:Robbins_Littlewood} for the last equality.
A similar calculation shows that the case $s_j=-q^{\half}$ of \eqref{conjid:fully} is equivalent to Corollary~\ref{cor:Littlewood}.

A consequence of Conjecture~\ref{conj:fully} would be a certain Littlewood identity for Hall--Littlewood polynomials~$P_{\lambda}(x_1,\ldots,x_n;q)$, which we recall are defined as
\begin{equation*}
	(1-q)^n \prod_{r \ge 0} \frac{1}{(q;q)_{m_r(\lambda)}} \frac{\asym_{x_{1},\ldots,x_{n}}\left[\prod\limits_{1\le i<j\le n}(x_i- q x_j)\prod\limits_{i=1}^{n}x_{i}^{\lambda_{i}}\right]}{\prod\limits_{1\le i<j\le n}(x_{i}-x_{j})}
\end{equation*}
and hence
\begin{equation*}
	\left. 
	P_{\lambda}(x_1,\ldots,x_n;q) = \prod_{r \ge 0} \frac{1}{(q;q)_{m_r(\lambda)}} F_{\lambda}(x_1,\ldots,x_n) \right|_{\xi_i=1,s_j=0}.
\end{equation*}
Conjecture~\ref{conj:fully} would thus imply the following Littlewood identity for Hall--Littlewood polynomials
\begin{multline}\label{eq:new_Hall_Littlewood} 
	\sum_{\lambda_1 \ge \lambda_2 \ge \ldots \ge \lambda_n \ge 0}  \prod_{r \ge 0} (-q^{\half};q^{\half})_{m_r(\lambda)}  P_{\lambda}(x_1,\ldots,x_n;q)
	= \prod_{i=1}^{n} \frac{1+q^{\half}}{1- x_i}
	\prod_{1 \le i < j \le n} \frac{1-q  x_i x_j}{x_i-x_j} \\
	\times \pf_{\chi_{\mathrm{even}}(n) \le i < j \le n} \left(
	\begin{cases}
		1, & i=0,\\
		\frac{(x_i-x_j)((1+q)(1- q^{\half} x_i x_j) + (x_i + x_j)(q^{\half}- q))}{(1+q^{\half})(1- x_i x_j)(1- q x_i x_j)},& i\ge 1,
	\end{cases} \right)
\end{multline}
where we have used 
\begin{equation*}
	\frac{(q;q)_{m}}{(q^{\half};q^{\half})_{m}}  = (-q^{\half};q^{\half})_{m}.
\end{equation*}

Note that the left-hand side of \eqref{eq:new_Hall_Littlewood} is very similar to the left-hand side of the following Littlewood identity for Hall--Littlewood polynomials by Kawanaka \cite[(1.21)]{Kaw99}, which was essentially proved in \cite{Kaw91}:
\begin{equation}\label{eq:Kawanaka} 
	\sum_{\lambda_1 \ge \lambda_2 \ge \ldots \ge \lambda_n \ge 1}  \prod_{r \ge 1} (-q^{\half};q^{\half})_{m_r(\lambda)}  P_{\lambda}(x_1,\ldots,x_n;q)
	= \prod_{i=1}^{n} \frac{1+q^{\half} x_i}{1- x_i}
	\prod_{1 \le i < j \le n} \frac{1-q  x_i x_j}{1-x_i x_j}.
\end{equation}
These identities differ in that the former identity sums over partitions allowing parts equal to zero, whereas Kawanaka's identity sums over partitions with strictly positive parts.

Now, we present an even more general conjectural identity that involves an additional parameter~$\gamma$, compare
with Gavrilova's identity \cite[(1)]{Gav23}; it also includes Kawanaka's identity as a special case.

\begin{conj}
	\label{genconj}
	Let $n$ be a positive integer and $p$ a non-negative integer. Setting  
	$\xi_j=1$ for all $j$ in $\F_{\lambda}(u_1,\ldots,u_n)$, we conjecture
	\begin{multline*} 
		\sum_{\lambda_1 \ge \lambda_2 \ge \ldots \ge \lambda_n \ge 0} 
		\frac{(-\gamma q^{\half};q^{\half})_{m_0(\lambda)}}{(q;q)_{m_0(\lambda)}} (-\gamma^{-1}s_0;q^{\half})_{m_0(\lambda)} \prod_{r \ge 1} \frac{(-s_r;q^{\half})_{m_r(\lambda)}}{(q^{\half};q^{\half})_{m_r(\lambda)}}  \F_{\lambda}(u_1,\ldots,u_n)\\
		=  \prod_{i=1}^n \frac{1+q^{\half}}{1-u_i}
		\prod_{1 \le i < j \le n} \frac{1-q  u_i u_j}{u_i-u_j}\\ 
		\times \pf_{\chi_{\mathrm{even}}(n) \le i < j \le n} \left(
		\begin{cases}
			1 + (\gamma-1) \frac{(q^{\half} - \gamma^{-1} s_0)(1-u_j)}{(1+q^{\half}) (1- s_0 u_j)}, & i=0,\\
			\frac{(u_i-u_j)((1+q)(1- q^{\half} u_i u_j) + (u_i + u_j)(q^{\half}- q)  +
				(\gamma-1) f(u_i,u_j))}{(1+q^{\half})(1- u_i u_j)(1- q u_i u_j)},& i\ge 1,
		\end{cases} \right) 
	\end{multline*}
	where 
	\begin{multline*} 
		f(u,v) \coloneq 
		\frac{(q^{\half} - \gamma^{-1} s_0) (1-u v)}{(1+q^{\half})(1-s_0 u) (1-s_0 v)}
		\left( (1+ q^{\half})(1+ \gamma^{-1} s_0) (1 + q \gamma u v) \right. \\
		\left. +(1 + \gamma)(q - \gamma^{-1} s_0)(1 - q^{\half} u v) - (1 + \gamma) q^{\half} (q^{\half} + \gamma^{-1} s_0) (u + v) \right) 
	\end{multline*} 	
	as well as $s_j=s$ for all $j \ge p$ and  $u_i = \frac{s+x_i}{1 + s x_i}$ for all $i$, and both sides are considered as power series in $x_1,x_2,\ldots,x_n$.
\end{conj}

This conjecture has also been checked for $n=1$, for $n=2$ if $p=0,1$, and for $n=3,4$ if $p=0$.

If we first set $s_0=0$ and then $\gamma=0$, the right-hand side of the identity in Conjecture~\ref{genconj} simplifies to 
\begin{equation*}
	\prod_{i=1}^{n} \frac{1+q^{\half} u_i}{1-u_i} \prod_{1 \le i < j \le n} \frac{1-q u_i u_j}{1- u_i u_j}.
\end{equation*}
Further setting $s_r=0$ for all $r$ leads to Kawanaka's identity \eqref{eq:Kawanaka}.

Finally, we present the following related Littlewood identity for fully inhomogeneous spin Hall--Littlewood polynomials, which we were unable to find in the literature. It was discovered in collaborative work with Moritz Gangl and its proof will appear in the above mentioned forthcoming paper related to Conjecture~\ref{genconj}.

\begin{thm} 
	Let $n$ be a positive integer and $p$ a non-negative integer. Setting  
	$\xi_j=1$ for all $j$ in $\F_{\lambda}(u_1,\ldots,u_n)$, we have 
	\[
	\sum_{\lambda_1 \ge \lambda_2 \ge \ldots \ge \lambda_n \ge 0} 
	\prod_{r \ge 0} \frac{(-s_r;q)_{m_r(\lambda)}}{(q;q)_{m_r(\lambda)}} \F_{\lambda}(u_1,\ldots,u_n)   =   
	\prod_{i=1}^{n} \frac{1}{1- u_i}
	\prod_{1 \le i < j \le n} \frac{1-q  u_i u_j}{1-u_i u_j}, 
	\]
	where $s_j=s$ for all $j \ge p$ and $u_i = \frac{s+x_i}{1 + s x_i}$ for all $i$, and both sides are considered as power series in $x_1,x_2,\ldots,x_n$.
\end{thm} 

It is a generalization of (1.4) in \cite{Fis22}, which is obtained by setting $s_x=-q^{-1}$and $u_i=\frac{-1+q x_i}{q-x_i}$, and identifying $w=-q-1-q^{-1}$ in \cite[(1.4)]{Fis22}.

\section*{Acknowledgements}

We thank the anonymous referees for their helpful comments and for pointing out the relation to fully inhomogeneous spin Hall--Littlewood symmetric rational functions. We also thank Moritz Gangl for helpful discussions.
	
\bibliographystyle{alphaurl}
\bibliography{ASMLittlewood.bib}

\end{document}